\documentclass[12pt]{article}
\usepackage{latexsym,bm}
\usepackage[tbtags]{amsmath}
\usepackage{amssymb}

\usepackage{amsmath,amsfonts,amsthm,amssymb}

\usepackage{indentfirst}
\usepackage{paralist}
\usepackage[english]{babel}
\usepackage{amsmath,amssymb}

\usepackage{cite}
\usepackage{tabularx}
\usepackage{booktabs}
\usepackage{dcolumn}

\usepackage{multirow}

\usepackage{graphicx}

\usepackage{wrapfig}
\usepackage{multicol}
\usepackage{verbatim}
\def\R{{\mathbb R}}
\renewcommand{\a}{\alpha}
\renewcommand{\b}{\beta}

\renewcommand{\l}{\lambda}
\renewcommand{\o}{\omega}
\renewcommand{\t}{\theta}
\newcommand{\<}{\langle}
\renewcommand{\>}{\rangle}

\def\P{{\mathbb P}}
\newtheoremstyle{theorem}
{10pt}
{10pt}
{\sl}
{}
{\bf}
{. }
{ }
{}
\theoremstyle{theorem}

\newtheorem{thm}{Theorem}[section]
\newtheorem{cor}{Corollary}[section]

\newtheorem{de}{Definition}[section]
\newtheorem{lem}{Lemma}[section]

\newtheorem{rem}{Remark}[section]

\newtheorem{ex}{Example}[section]
\renewcommand{\b}{\beta}
\newcommand{\g}{\gamma}
\newcommand{\mcD}{\mathcal D}
\newcommand{\mcA}{\mathcal{A}}
\newcommand{\mcB}{\mathcal{B}}
\newcommand{\mcF}{\mathcal{F}}
\newcommand{\F}{\mathcal{F}}
\numberwithin{equation}{section}

\newtheoremstyle{defi}
{10pt}
{10pt}
{\rm}
{}
{\bf}
{. }
{ }
{}
\theoremstyle{defi}

\allowdisplaybreaks \textheight 21 true cm
\usepackage{epsfig,graphicx,epsf,wrapfig}
\usepackage{dcolumn}
\usepackage{bm}
\begin{document} 


\title{{\large  \textbf{Random attractors for locally monotone  stochastic partial differential equations with linear multiplicative fractional noise \thanks{This work
				is supported in part by a NSFC Grant No. 12171084 and the fundamental Research
				Funds for the Central Universities No. 2242022R10013.}   }}
\author{Qiyong Cao
\footnote{Email: xjlyysx@163.com.} \  \  Hongjun Gao \footnote{Correspondence, Email: hjgao@seu.edu.cn}
\\
\\
\footnotesize{ School of Mathematics,
	Southeast University, Nanjing 211189, P. R. China} 
}
 }  

 \date{}
\maketitle 
\footnotesize
\noindent \textbf{Abstract~~~} In this paper, we consider the random attractors for a class of locally monotone stochastic partial differential equations perturbed
by the linear multiplicative fractional Brownian motion with Hurst index $H\in(\frac{1}{2},1)$. We obtain the random attractors or $\mathcal{D}$-pullback random attractors for these
systems and some examples are given in this paper.
\\[2mm]
\textbf{Key words~~~} Random dynamical systems; random attractors; 2D Navier-Stokes equation; Porous media equation; fractional Brownian motion.
\\
\textbf{2020 Mathematics Subject Classification~~~}37L30; 37L55; 60G22; 35R60

\section{Introduction}
\setcounter{equation}{0}
The  random attractors  for stochastic partial  differential  equations(SPDEs)  driven by   white noise, colored noise or L\'{e}vy noise  have been intensively studied(Refs.
\cite{MR1451294,MR1305587,schmalfuss1992backward,MR3959496,MR3960498,MR4022971,MR4093237,MR2468738,MR2498851,MR2927390,MR2775822,MR4097578,MR3787205,MR2780247,MR2531632,MR1943557,MR3424621,MR4097253,MR2812588,MR3178475,MR3053476,MR3027636,MR3612976,MR2510419}
et al.).  Note that the references \cite{MR1451294,MR1305587,schmalfuss1992backward} provide the basic theory of random attractors. Compared with the previous results we mentioned,
the random attractors of the SPDEs driven by the fractional Brownian motion(fBm) are relatively few,  and the fBm in complex models always  is additive, such as, Refs.
\cite{MR3225217,MR2946314,MR2812588,MR3047956} et al. Due to this fact, we mainly focus on the multiplicative fractional noise in this paper.  For the multiplicative fBm,  there are
some results on random attractors, such as, Refs.  \cite{MR3996912,MR2738732,MR3226746,MR4385780,MR4608383,MR4594461,MR4383300,cao2022wongzakai}.  On the one hand,  although the
methods in Refs. \cite{MR3996912,MR4594461} can be  applied to some systems driven by the linear multiplicative fBm, more restricts on the drift terms are imposed, such as, globally
Lipschitz with small Lipschitz constant.  On the other hand, a class of  evolution systems driven by the  nonlinear multiplicative fBm are discussed in Refs.
\cite{MR3226746,MR4385780,MR4608383,MR4383300,cao2022wongzakai,MR2738732}. Note that  all these results require that the drift terms are globally Lipschitz and the diffusion terms
have some kind of boundedness.      In addition, the methods in Refs. \cite{MR3226746,MR4608383,MR4383300,cao2022wongzakai,MR2738732} require that the drift terms are globally
Lipschitz  and the Lipschitz constants are very small.   It is worth mentioning that the method in \cite{MR4385780} can be applied to  the systems which the drift term is locally
Lipschitz and dissipative,  however it can not be applicable to  the infinite-dimensional systems. So the methods in those paper we mentioned can not be applied to the complex model
in infinite-dimensional systems driven by the multiplicative fBm, even in the case of  linear multiplicative fBm.
In short, the infinite-dimensional complex systems with the  multiplicative fBm  are still an open question.  Therefore,  it is a challenging task to deal with the general drift term
in SPDEs driven by the multiplicative  fBm, such as, stochastic Navier-Stokes model,  stochastic Cahn-Hilliard model, et al. As the first step of this field, our aim is to consider
the complex systems driven by the linear multiplicative fBm.

It is very important to consider the well-posedness of SPDEs driven by the multiplicative fBm, it provides a foundation for the study of the dynamics of the system.  The global solution of SPDEs driven by
the multiplicative fBm are considered  in  Refs. \cite{MR3072986,MR4431448,MR4097587},  their methods are not applied to the  model with local Lipschitz drift and driven by  the multiplicative fBm.
Motivated by the results we mentioned, we plan to construct the variational method for SPDEs  driven by the linear multiplicative  fBm with Hurst index $H>\frac{1}{2}$ and its random
attractors.

More precisely, we consider the autonomous SPDEs of the form
\begin{equation}
	du_{t}=A(u_{t})dt+\beta u_td\omega_{t}\label{eqn:SPDE1}
\end{equation}
and the nonautonomous SPDEs of the form
\begin{equation}\label{eqn:SPDE2}
	du_{t}=A(t,u_{t})dt+\beta u_td\omega_{t},
\end{equation}
where $\omega_{t}$ is a realized version of the fBm $B^H_t(\omega)$ with  Hurst index $H\in(\frac{1}{2},1)$, $\beta\in\mathbb{R}$ and the operators $A$ are locally monotone in the
equation \eqref{eqn:SPDE1} and equation \eqref{eqn:SPDE2} (cf.\ $(A2)$ and $(A2^\prime)$ below) with respect to a Gelfand triple $V\subseteq H \subseteq V^*$. The variational
framework depends on the concept of \textit{locally} monotone operators. We extend the previous results(see Refs. \cite{MR2812588,MR4097253,MR3027636}), which  noises do not contain
the linear multiplicative fBm. Furthermore, our method can be applied to some complex models, such as, Porous media equation, $2D$ stochastic Navier-Stokes equation, more details can
be found in Section \ref{example}.

Our the main strategy to complete this work as follows: The well-posedness of the equation \eqref{eqn:SPDE1} and the equation \eqref{eqn:SPDE2} depend on the transformation we
introduced in Section \ref{sec:setup}, based on the transformation, we can transform the equations  \eqref{eqn:SPDE1} and \eqref{eqn:SPDE2} into  a random partial differential
equation(random PDE). According to the equivalence of the solution(Lemma \ref{equvilence} and Lemma \ref{equvilence1} below), we can get the variational solution of the equations
\eqref{eqn:SPDE1} and \eqref{eqn:SPDE2}.
For the random attractors of the systems \eqref{eqn:SPDE1} and \eqref{eqn:SPDE2}, we  can use the classical  framework, namely the  existence of bounded absorbing set  in $H$ and the
precompactness of random dynamical system to construct the random attractors.

In this paper, the main difficulty is to get the appropriate H\"{o}lder regularity of the solution to complete the well-posedness of the equations \eqref{eqn:SPDE1}  and
\eqref{eqn:SPDE2}. Due to the conditions $(A4)$ and $(A4^\prime)$, it means that the H\"{o}lder regularity of the solution must be less than $1/2$. Motivated by the property of Young
integral  and the regularity in Refs. \cite{MR3996912,MR4594461}, we make the H\"{o}lder regularity of noise  better  than the H\"{o}lder regularity of the solution. It is very
important to get the random attractors. For the rough noise ($\frac{1}{3}<H\leq \frac{1}{2}$), as we have described, it is very diffcult to get the wanted H\"{o}lder regularity. We
shall try this work in the future.

The generality of the framework of locally monotone SPDEs \eqref{eqn:SPDE1} and \eqref{eqn:SPDE2} driven by the linear multiplicative fBm has several  advantages: Firstly, compared
with  the semigroup method and variational method(Refs. \cite{MR3072986,MR4431448,MR4097587,MR4254494,MR3918521}),  our method  is more concise. For example, we do not need to obtain
the  existence of solution via the approximation method.  Secondly, we can deal with more complicated models as in Refs. \cite{MR2812588,MR4097253,MR3027636}.  Thirdly, it is not
necessary to impose more strictly condition as in Refs. \cite{MR3226746,MR4385780,MR4608383,MR4383300,cao2022wongzakai,MR2738732}, such as, small Lipschitz conditions on the drift
term.
Finally, we do not need  to construct an O-U process, and  consider the strong  monotone conditions  which  are imposed for O-U process in Refs.
\cite{MR2812588,MR4097253,MR3027636}.

The paper is organized as follows: In Section \ref{sec:setup},  we will give some  assumptions on the coefficients and recall Young integral. Section \ref{well-posedness} is devoted
to construct the global variational solutions of the equations \eqref{eqn:SPDE1} and \eqref{eqn:SPDE2}. In Section \ref{sec:construction}, we first construct the random dynamical
system $\varphi$ for the system \eqref{eqn:SPDE1} and the continuous cocycle $\Phi$ for the system \eqref{eqn:SPDE2}. Combining with the existence of a random bounded absorbing set
and the precompactness of $\varphi$ and $\Phi$,  we can obtain the existence of a random attractor. Some examples to the SPDEs are given in Section \ref{example}.  In Appendix
\ref{app:rds}, we will recall the theory of random dynamical systems.
\section{Preliminaries}\label{sec:setup}

Let $(\Omega,\mathcal{F},\mathbb{P},\{\theta_{ t}\}_{t\in\R})$ be a metric dynamical system for the fractional Brownian motion, it is introduced in Section \ref{sec:construction}.
Let $T>0$ be any positive real number. Let $(H,\<\cdot,\cdot\>_{H})$ be a real separable Hilbert space, identified with its dual space $H^{*}$ by the Riesz isomorphism.
Let $V$ be a real reflexive Banach space, it is continuously and densely embedded into $H$.
In particular, there is a constant $\l>0$ such that $ \l \|v\|_{H}^{2}\le \|v\|_{V}^{2}$ for all $v\in V$. Then we have the following Gelfand triple
\[
V\subseteq H\equiv H^{*}\subseteq V^{*}.
\]
If $_{V^{*}}\<\cdot,\cdot\>_{V}$ denotes the dualization between $V$ and its dual space $V^{*}$, then
\[
_{V^{*}}\<u,v\>_{V}=\<u,v\>_{H},\quad\forall u\in H,v\in V.
\]

As mentioned in the introduction, we consider the following SPDEs of the form
\begin{equation}
	du_{t}=A(u_{t})dt+\beta u_{t}d\omega_{t},\label{eqn:spde2}
\end{equation}
and
\begin{equation}
	du_{t}=A(t,u_{t})dt+\beta u_{t}d\omega_{t},\label{eqn:spde3}
\end{equation}
where $A: V\to V^{*}$ is  $\mathcal{B}(V)/\mathcal{B}(V^*)$  measurable in the equation \eqref{eqn:spde2}, $A: \mathbb{R}\times V\to V^{*}$ is  $(\mathcal{B}(\mathbb
R)\times\mathcal{B}(V))/\mathcal{B}(V^*)$  measurable in the equation \eqref{eqn:spde3},  and $\beta\in\mathbb{R}$, $\omega$ is the realization of the fractional Brownian motion
$B^{H}(t,\omega):=\omega_t$ with Hurst index $H>\frac{1}{2}$, and $t\in\mathbb{R},\omega\in\Omega$.

For the coefficient $A$ in the equation \eqref{eqn:spde2}. Suppose that for some $\alpha\ge 2$ and $\varpi\ge0$, there exist some constants $C,K\ge0$,$\gamma>0$
such that the following conditions hold for all $v,v_{1},v_{2}\in V$:
\begin{enumerate}
	\item [$(A1)$](Hemicontinuity) The map $s\mapsto{}_{V^{*}}\<A(v_{1}+sv_{2}),v\>_{V}$ is continuous on $\mathbb{R}$.
	\item [$(A2)$](Local monotonicity)
	\[
	2{}_{V^{*}}\<A(v_{1})-A(v_{2}),v_{1}-v_{2}\>_{V}\le\left(C+\eta(v_{1})+\rho(v_{2})\right)\|v_{1}-v_{2}\|_{H}^{2},
	\]
	where $\eta,\rho:V\rightarrow\R_{+}$ are locally bounded measurable functions.
	\item [$(A3)$](Coercivity)
	\[
	2{}_{V^{*}}\<A(v),v\>_{V}\le-\gamma\|v\|_{V}^{\alpha}+K\|v\|_{H}^{2}+C.
	\]
	
	\item [$(A4)$](Growth) 
	
	\[
	\|A(v)\|_{V^{*}}^{\frac{\a}{\a-1}}\le C(1+\|v\|_{V}^{\alpha})(1+\|v\|_{H}^{\varpi}).
	\]
	
\end{enumerate}
In addition, for the nonautonomous operator $A(t,u_t)$ in the equation \eqref{eqn:spde3}, we impose the similar conditions as follows:

Assume that for some $\alpha\ge 2,\varpi\ge0,\tau\in\R$, there exist some constant $c,C \ge 0$ and functions $f,g\in L^{1}([\tau,\tau+T];\mathbb{R})$ such that the following
conditions hold for all $t\in[\tau,\tau+T]$ and $v,v_{1},v_{2}\in V$:
\begin{enumerate}
	\item [$(A1^\prime)$] (Hemicontinuity) The map $s\mapsto{}_{V^{*}}\<A(t,v_{1}+sv_{2}),v\>_{V}$ is continuous on $\mathbb{R}$.
	\item [$(A2^\prime)$] (Local monotonicity)
	\[
	2{}_{V^{*}}\<A(t,v_{1})-A(t,v_{2}),v_{1}-v_{2}\>_{V}\le\left(f(t)+\eta(v_{1})+\rho(v_{2})\right)\|v_{1}-v_{2}\|_{H}^{2},
	\]
	where $\eta,\rho:V\rightarrow[0,+\infty)$ are measurable and locally bounded functions.
	\item [$(A3^\prime)$] (Coercivity)
	\[
	2{}_{V^{*}}\<A(t,v),v\>_{V}\le-c\|v\|_{V}^{\alpha}+g(t)\|v\|_{H}^{2}+f(t)
	\]
	\item [$(A4^\prime)$] (Growth)
	\[
	\|A(t,v)\|_{V^{*}}^{\frac{\alpha}{\alpha-1}}\le C\bigg(f(t)+\|v\|_{V}^{\alpha}\bigg)\bigg(1+\|v\|_{H}^{\varpi}\bigg).
	\]
\end{enumerate}
Let us now define the variational solution to \eqref{eqn:spde2} and \eqref{eqn:spde3}.
\begin{de}
	\label{def:soln_pathw} A continuous $H$-valued process $\{u(t,\omega,u_{0})\}_{t \in [0,T]}$  with initial condition $u_{0}$ at time $0$  is a variational  solution to
	\eqref{eqn:spde2} if for all $\o\in\Omega$, $u(t,\o,u_{0})\in L^{\a}([0,T];V)$ and for any $v\in V, t\in[0,T],T>0$, the following holds
	\begin{align*}
		\<u(t,\o,u_{0}),v\>_{H}=\<u_{0},v\>_{H}+\int_{0}^{t} {}_{V^{*}}\<A(u(r,\o,u_0),v\>_{V}dr\\
		+\int_{0}^{t}\beta \<u(r,\omega,u_0),v\>_{H}d\o_{r}.
	\end{align*}
	Further, the continuous $H$-valued process $\{u(t,\tau,\omega,u_\tau)\}_{t\geq \tau}$ with initial data $u_\tau$ at time $\tau\in\R$ is a variational solution to
\eqref{eqn:spde3} if for all $\omega\in\Omega,u(t,\tau,\omega,u_\tau)\in L^{\alpha}([\tau,\tau+T];V)$ and for any $v\in V,t\geq \tau,T>0$, the following holds
	\begin{align*}
		\<u(t,\tau,\o,u_{\tau}),v\>_{H}=\<u_{\tau},v\>_{H}+\int_{\tau}^{t} {}_{V^{*}}\<A(u(r,\tau,\o,u_0),v\>_{V}dr\\
		+\int_{\tau}^{t}\beta \<u(r,\tau,\omega,u_0),v\>_{H}d\o_{r}.
	\end{align*}
\end{de}
\begin{rem}
	The integrals $\int_{0}^{t}\beta \<u(r,\omega,u_0),v\>_{H}d\o_{r}$ and   $\int_{0}^{t}\beta \<u(r,\tau,\omega,u_0),v\>_{H}d\o_{r}$ are Young integral($H>\frac{1}{2}$), then for
every fixed $\omega\in\Omega$ and $\tau\in\mathbb{R}$, it means that $\<u(t,\omega,u_0),v\>_{H}$ and $\<u(t,\tau,\omega,u_\tau),v\>_{H}$ are locally H\"{o}lder continuity with
respect to variable $t\in[0,\infty)$ and $t\in[\tau,\infty)$, respectively.
\end{rem}
\begin{thm}{\cite{MR4174393}}\label{Young est}
	Let $V$ and $W$ be  two Banach spaces.	For any  $\zeta$-H\"{o}lder path $x_t$  in $V$ and $\xi$-H\"{o}lder path $y_t$ in $\mathcal{L}(V,W)$, the space
	of bounded linear operators from $V$ into some Banach space $W$. Then Young integral can be defined as follows:
	\[
	\int_s^t y_u dx_u := \lim \limits_{|\Pi| \to 0} \sum_{[u,v] \in \Pi} \Big( y_{u}  \otimes  x_{u,v}  \Big),
	\]
	where the limit is taken on all the finite partition $\Pi$ of $[s,t]$ with $|\Pi| := \displaystyle\max_{[u,v]\in \Pi} |v-u|$, and $x_{u,v}=x_v-x_u$. In addition,
	the integral has the following estimate:
	\begin{eqnarray*}
		\left\|\int_{s}^{t}x_rdy_r-x_sy_{s,t}\right\|\leq C\|x\|_{\zeta,[s,t]}\|y\|_{\xi,[s,t]}(t-s)^{\zeta+\xi},\quad  s,t\in\mathbb{R},
	\end{eqnarray*}
	where $\zeta+\xi>1$, $y_{s,t}=y_t-y_s$ and the constant  $C$ depends on  $\zeta,\xi$.
\end{thm}
In order to solve  SPDEs \eqref{eqn:spde2} and \eqref{eqn:spde3}, we need to introduce a transformation.  Let $z_t(\omega)=e^{-\beta\omega_{t}}$ and
$z_t^{-1}(\omega)=e^{\beta\omega_t}$. Since the fractional Brownian motion $\omega$ has $\gamma_1$-H\"{o}lder continuous version, where $\frac{1}{2}<\gamma_1<H$, i.e. $\omega\in
C^{\gamma_1}([\tau,\tau+T];\mathbb{R})$ for any $T> 0, \tau\in\mathbb{R}$. Then  we have the following lemma.
\begin{lem}\label{Trans noise}
	Let $\b\in\mathbb{R},\tau\in\mathbb{R},T\geq 0$, then $z_t$ and $z_t^{-1}$ satisfy the following stochastic differential equations
	\begin{align*}
		dz_t=-\beta z_t d\omega_t,\quad t\in[\tau,\tau+T]
	\end{align*}
	and
	\begin{align*}
		dz_t^{-1}=\beta z_t^{-1} d\omega_t,\quad t\in[\tau,\tau+T],
	\end{align*}
	respectively.
\end{lem}
\begin{proof}
	We only check that the stochastic differential equation for $z_t(\omega)=e^{-\beta \omega_t}$, the same method can be applied to $z_t^{-1}(\omega)=e^{\beta \omega_t}$. For any
$t\in[\tau,\tau+T]$, let $\mathcal{P}(\tau,t)$ be an arbitrary partition of the interval $[\tau,t]$, we have
	\begin{align*}
		e^{-\beta\omega_t}-e^{-\beta\omega_{\tau}}&=\sum_{[u,v]\in\mathcal{P}(\tau,t)}(e^{-\beta\omega_v}-e^{-\beta\omega_u})\\
		&=\sum_{[u,v]\in\mathcal{P}(\tau,t)}(-\beta e^{-\beta\omega_u}(\omega_v-\omega_u))\\
		&\quad+\sum_{[u,v]\in\mathcal{P}(\tau,t)}\left(\frac{1}{2}{\beta}^2e^{-\beta(\theta\omega_u+(1-\theta)\omega_v))}(\omega_v-\omega_u)^2\right),
	\end{align*}
	where $\theta\in [0,1]$. Note that $\omega_t$ is a continuous function  with respect to variable $t$, then there exists a  $u^{\prime}\in [u,v]$ such that
	$\theta\omega_u+(1-\theta)\omega_v=\omega_{u^\prime}$. We claim that  there exists a constant $C:=C(\beta,\omega,\tau,T)>0$ such that
	\begin{align*}
		\left|\frac{1}{2}{\beta}^2e^{-\beta\omega_{u^\prime}}(\omega_v-\omega_u)^2\right|\leq C(v-u)^{2\gamma_1}.
	\end{align*}
	Indeed, if  $\tau+T\leq 0$, then
	\begin{align*}		\left|\frac{1}{2}{\beta}^2e^{-\beta\omega_{u^\prime}}(\omega_v-\omega_u)^2\right|\leq\frac{1}{2}{\beta}^2e^{|\beta|\|\omega\|_{\gamma_1,[\tau,0]}|\tau|^{\gamma_1}}\|\omega\|^2_{\gamma_1,[\tau,\tau+T]}(v-u)^{2\gamma_1}.
	\end{align*}
	If $\tau+T>0,\tau<0$, then
	\begin{align*} \left|\frac{1}{2}{\beta}^2e^{-\beta\omega_{u^\prime}}(\omega_v-\omega_u)^2\right|\leq\frac{1}{2}{\beta}^2e^{|\beta|\|\omega\|_{\gamma_1,[\tau,\tau+T]}T^{\gamma_1}}\|\omega\|^2_{\gamma_1,[\tau,\tau+T]}(v-u)^{2\gamma_1}.
	\end{align*}
	If $\tau+T>0,\tau\geq 0$, then
	\begin{align*} \left|\frac{1}{2}{\beta}^2e^{-\beta\omega_{u^\prime}}(\omega_v-\omega_u)^2\right|\leq\frac{1}{2}{\beta}^2e^{|\beta|\|\omega\|_{\gamma_1,[0,\tau+T]}(T+\tau)^{\gamma_1}}\|\omega\|^2_{\gamma_1,[\tau,\tau+T]}(v-u)^{2\gamma_1}.
	\end{align*}
	Therefore, there exists a uniform constant $$C=\frac{1}{2}{\beta}^2\|\omega\|^2_{\gamma_1,[\tau,\tau+T]}\left(e^{|\beta|\|\omega\|_{\gamma_1,[\tau,0]}\tau^{\gamma_1}}\vee
e^{|\beta|\|\omega\|_{\gamma_1,[\tau,\tau+T]}T^{\gamma_1}}\vee e^{|\beta|\|\omega\|_{\gamma_1,[0,\tau+T]}(T+\tau)^{\gamma_1}} \right)$$
	for each $[u,v]\in\mathcal{P}(\tau,t)$. Due to the above discussions, we know that
	\begin{align*}
		\lim_{|\mathcal{P}(\tau,t)|\rightarrow
0}\sum_{[u,v]\in\mathcal{P}(\tau,t)}&\left(\frac{1}{2}{\beta}^2e^{-\beta(\theta\omega_u+(1-\theta)\omega_v))}(\omega_v-\omega_u)^2\right)\\
		&\leq \lim_{|\mathcal{P}(\tau,t)|\rightarrow 0}\sum_{[u,v]\in\mathcal{P}(\tau,t)}C(v-u)^{2\gamma_1}\\
		&\leq CT\lim_{|\mathcal{P}(\tau,t)|\rightarrow 0}|\mathcal{P}(\tau,t)|^{2\gamma_1-1}\\
		&=0.
	\end{align*}
	By Theorem  \ref{Young est},  we have
	\begin{align*}
		e^{-\beta\omega_t}-e^{-\beta\omega_{\tau}}&=\lim_{|\mathcal{P}(\tau,t)|\rightarrow 0}\sum_{[u,v]\in\mathcal{P}(\tau,t)}(e^{-\beta\omega_v}-e^{-\beta\omega_u})\\
		&=\lim_{|\mathcal{P}(\tau,t)|\rightarrow 0}\sum_{[u,v]\in\mathcal{P}(\tau,t)}(-\beta e^{-\beta\omega_u}(\omega_v-\omega_u))\\
		&=\int_{\tau}^{t}-\beta e^{-\beta \omega_r}d\omega_r,
	\end{align*}
	it means that $	dz_t=-\beta z_td\omega_t$.
\end{proof}
\begin{lem}\label{chain rule}
	Let $\omega_t$ be a real-valued fractional Brownian motion with $\gamma_1$-H\"{o}lder continuous version, and $X_t\in C^{\gamma_1}([\tau,\tau+T],\mathbb{R})$ and $Y_t\in
C^{\gamma_2}([\tau,\tau+T],\mathbb{R}^m)$ satisfy the following stochastic integral  equations
	\begin{align*}
		Y_t=\Gamma_t+\int_{\tau}^{t}g(s)\rm d\omega_s
	\end{align*}
	and
	\begin{align*}
		X_t=\Lambda_t+\int_{\tau}^{t}q(s)\rm d\omega_s,
	\end{align*}
	respectively, where $\Gamma$ and $\Lambda$ are drift terms and $\Gamma\in C^{\gamma_2}([\tau,\tau+T],\mathbb{R})$, $\Lambda\in C^{\gamma_1}([\tau,\tau+T],\mathbb{R}^m)$,  $g$ and
$q$  are $\gamma_2$-H\"{o}lder  continuous function and $\gamma_1$-H\"{o}lder  continuous function, where  $\frac{1}{2}<\gamma_1<H$ and $\gamma_2+\gamma_1>1$, respectively.  Then we
have
	\begin{align*}
		\rm d(XY)_t&=X_td Y_t+Y_td X_t\\
		&=(X_td \Gamma_t+Y_td\Lambda_t)+(X_tg_t+Y_tq_t)d \omega_t.
	\end{align*}
\end{lem}
\begin{proof}
	By the definition of Young integral, $X_t$ and $Y_t$ has the following increment forms:
	\begin{equation}\label{increment1}
		\begin{aligned}
			Y_{s,t}=\Gamma_{s,t}+g(s)\omega_{s,t}+o(|t-s|),
		\end{aligned}
	\end{equation}
	\begin{equation}\label{increment2}
		\begin{aligned}
			X_{s,t}=\Lambda_{s,t}+q(s)\omega_{s,t}+o(|t-s|).
		\end{aligned}
	\end{equation}
	According to \eqref{increment1} and \eqref{increment2}, we have
	\begin{align*}
		Y_{t}X_{t}-Y_{\tau}X_{\tau}&=\sum_{[u,v]\in\mathcal{P}(\tau,t)}(-Y_vX_{v,u}+Y_{u,v}X_u)\\
		&=\sum_{[u,v]\in\mathcal{P}(\tau,t)}-Y_v(\Lambda_{v,u}+q(v)\omega_{v,u}+o(|v-u|))\\
		&\quad+\sum_{[u,v]\in\mathcal{P}(\tau,t)}X_u(\Gamma_{u,v}+g(u)\omega_{u,v}+o(|u-v|)).
	\end{align*}
	Let $|\mathcal{P}(\tau,t)|\rightarrow 0$, then $\sum_{[u,v]\in\mathcal{P}(\tau,t)}o(|v-u|)\rightarrow 0$. Thus, by the definition of Young integral, we obtain
	\begin{align*}
		Y_{t}X_{t}-Y_{\tau}X_{\tau}&=\lim_{|\mathcal{P}(\tau,t)|\rightarrow 0}\sum_{[u,v]\in\mathcal{P}(\tau,t)}-Y_v(\Lambda_{v,u}+q(v)\omega_{v,u})\\
		&\quad+\lim_{|\mathcal{P}(\tau,t)|\rightarrow 0}\sum_{[u,v]\in\mathcal{P}(\tau,t)}X_u(\Gamma_{u,v}+g(u)\omega_{u,v})\\
		&=\lim_{|\mathcal{P}(\tau,t)|\rightarrow 0}\sum_{[u,v]\in\mathcal{P}(\tau,t)}-(Y_v-Y_u+Y_u)   (\Lambda_{v,u}+q(v)\omega_{v,u})\\
		&\quad+\lim_{|\mathcal{P}(\tau,t)|\rightarrow 0}\sum_{[u,v]\in\mathcal{P}(\tau,t)}X_u(\Gamma_{u,v}+g(u)\omega_{u,v})\\
		&=\int_{\tau}^{t}Y_rq_rd \omega_r+\int_{\tau}^{t}X_rg(r)d \omega_{r}+\int_{\tau}^{t}Y_rd \Lambda_r+\int_{\tau}^{t}X_rd\Gamma_r,
	\end{align*}
	it means that $	\rm d(XY)_t=X_td Y_t+Y_td X_t$.
\end{proof}
Let $\tilde{u}(t)=z_t(\omega)u(t)$, if $\<\tilde{u}(\cdot,\omega,u_0),v\>_{H}$ is a $\gamma_2$-H\"{o}lder continuous function on $[0,T]$ and $\gamma_1+\gamma_2>1$, then applying
Lemma \ref{chain rule} to $\<\tilde{u}(t,\omega,u_0),v\>_{H}$, we have
\[
\<\tilde{u}(t),v\>_{H}=\<u_0,v\>_{H}+\int_{0}^{t}e^{-\beta\omega_r}{}_{V^{*}}\<A(z_r^{-1}(\omega)\tilde{u}(r)),v\>_{V}dr,\quad \forall v\in V,
\]
it is the variational  formulation of the following random partial differential equations
\begin{equation}\label{random PDE}
	d\tilde{u}(t)=z_t(\omega)A(z_t^{-1}(\omega)\tilde{u}(t))dt
\end{equation}
with initial data $u_0$. Similarly, if $\<\tilde{u}(\cdot,\tau,\omega,u_\tau),v\>_{H}$ is a $\gamma_2$-H\"{o}lder continuous function on $[\tau,T+\tau]$ and $\gamma_2+\gamma_1>1$,
then applying Lemma \ref{chain rule} to $\<\tilde{u}(t,\tau,\omega,u_\tau),v\>_{H}$, we have
\[
\<\tilde{u}(t),v\>_{H}=\<\tilde u_\tau,v\>_{H}+\int_{\tau}^{t}e^{-\beta\omega_r}{}_{V^{*}}\<A(r,z_r^{-1}(\omega)\tilde{u}(r)),v\>_{V}dr,\quad \forall v\in V,
\]
it is the variational  formulation of the following random partial differential equations
\begin{equation}\label{random PDE1}
	d\tilde{u}(t)=z_t(\omega)A(t,z_t^{-1}(\omega)\tilde{u}(t))dt
\end{equation}
with initial data $\tilde u_\tau=e^{-\beta\omega_\tau}u_\tau$.
Therefore, we can give the following definition of the random partial differential equations \eqref{random PDE} and \eqref{random PDE1}.
\begin{de}\label{random-t}
	A continuous $H$-valued process $\{\tilde{u}(t,\omega,u_0)\}$ with initial data $u_0$ at time $0$ is a variational  solution of \eqref{random PDE} if  for all $\omega\in
\Omega,T>0$,  $\tilde{u}(t,\omega,u_0)\in L^{\a}([0,T];V)$ and for any $v\in V, t\in[0,T]$, the following relation holds
	\[
	\<\tilde{u}(t,\o,u_0),v\>_{H}=\<u_0,v\>_{H}+\int_{0}^{t}e^{-\beta\omega_r}{}_{V^{*}}\<A(z_r^{-1}(\omega)\tilde{u}(r,\o,u_0)),v\>_{V}dr.
	\]
	In addition, the continuous $H$-valued process  $\{\tilde{u}(t,\tau,\omega,\tilde u_\tau)\}_{t\geq \tau}$ with initial data $\tilde u_\tau$ at time $\tau\in\mathbb{R}$ is a
variational  solution of \eqref{random PDE1} if  for all $\omega\in \Omega,T>0$, $\tilde{u}(t,\tau,\omega,\tilde u_\tau)\in L^{\a}([\tau,\tau+T];V)$ and for any $v\in V,
t\in[\tau,\tau+T]$, the following identity holds
	\[
	\<\tilde{u}(t,\tau,\o,u_0),v\>_{H}=\<\tilde u_\tau,v\>_{H}+\int_{0}^{t}e^{-\beta\omega_r}{}_{V^{*}}\<A(r,z_r^{-1}(\omega)\tilde{u}(r,\o,u_0)),v\>_{V}dr.
	\]
\end{de}
Based on the above transformation, if $u(t)$ is the variational solution of the system  \eqref{eqn:spde2}, then  $\tilde{u}(t)$ is the variational solution of \eqref{random PDE} and
vice versa.
\begin{lem}\label{equvilence}
	Let $\tilde{u}(t)=z_t(\omega)u(t)$, $u_0\in H$. Under the assumptions (A1)-(A4), the following statements are equivalent.
	\begin{description}
		\item[(i)] $\tilde{u}\in C([0,T];H)\cap L^{\alpha}([0,T];V)$ is the variational solution of \eqref{random PDE} for each fixed $\o\in\Omega$, and $\<\tilde{u},v\>\in
C^{\gamma_2}([0,T];\mathbb{R})$ for any $v\in V$ and $\omega\in \Omega$, where $\gamma_1+\gamma_2>1,\gamma_2\leq \frac{1}{\alpha}\wedge \gamma_1$.
		\item [(ii)] $u\in C([0,T];H)\cap L^{\alpha}([0,T];V)$ and $\<u,v\>_H\in C^{\gamma_2}([0,T];\mathbb{R})$ for $v\in V$ is the variational solution of \eqref{eqn:spde2} for
each fixed $\omega\in\Omega$, where $\gamma_1+\gamma_2>1$.
	\end{description}
\end{lem}
\begin{proof}
	We first assume that the statement \textbf{(i)} is correct. It is directly  from $u(t)=z_t^{-1}(\omega)\tilde{u}$ to know that $u\in C([0,T];H)\cap L^{\alpha}([0,T];V)$.
	
	Due to $\o\in C^{\gamma_1}([0,T],\mathbb{R}),\gamma_1\in(\frac{1}{2},H)$, for any $v\in V, \gamma_1+\gamma_2>1,\gamma_2\leq \frac{1}{\alpha}\wedge \gamma_1$, it suffices to check
that $\<u,v\>_{H}\in C^{\gamma_2}([0,T],\mathbb{R})$ . Indeed, for any $s,t\in[0,T],\o\in\Omega$, we have
	\begin{equation}\label{eqn:holder-e}
		\begin{aligned}
			&|\<u(t,\omega,u_0),v\>_{H}-\<u(s,\omega,u_0),v\>_{H}|\\
			&=|z_t^{-1}(\omega)\<\tilde{u}(t,\omega,u_0),v\>_{H}-z_s^{-1}(\omega)\<\tilde{u}(s,\omega,u_0),v\>_{H}|\\
			&\leq |(z_t^{-1}(\omega)-z_s^{-1}(\omega))\<\tilde{u}(t,\omega,u_0),v\>_{H}|\\
			&\quad+z_s^{-1}(\omega)|\<\tilde{u}(t,\omega,u_0),v\>_{H}-\<\tilde{u}(s,\omega,u_0),v\>_{H}|.
		\end{aligned}
	\end{equation}
	For the first term on the right hand of  \eqref{eqn:holder-e}, by the mean value theorem, we have
	\begin{equation}\label{eqn:holder-e-1}
		\begin{aligned}
			&|(z_t^{-1}(\omega)-z_s^{-1}(\omega))\<\tilde{u}(t,\omega,u_0),v\>_{H}|\\
			&\leq e^{|\beta|\sup_{t\in[0,T]}|\omega_t|}|\omega_t-\omega_{s}||\<\tilde{u}(t,\omega,u_0),v\>_{H}|\\
			&\leq C(\omega,T,\beta)|t-s|^{\gamma_1}\|v\|_{H}\sup_{t\in[0,T]}\|\tilde{u}(t,\o,u_0)\|_{H}\\
			&\leq C(\omega,T,\beta,\lambda)|t-s|^{\gamma_1}\|v\|_{v}\sup_{t\in[0,T]}\|\tilde{u}(t,\o,u_0)\|_{H}.
		\end{aligned}
	\end{equation}
	For the second term on the right hand of \eqref{eqn:holder-e}, according to Definition \ref{random-t} and assumption (A4), we have
	\begin{align}\label{eqn:holder-e-2}	
		&z_s^{-1}(\omega)|\<\tilde{u}(t,\omega,u_0),v\>_{H}-\<\tilde{u}(s,\omega,u_0),v\>_{H}|\nonumber\\
		&\leq C(\omega,T,\beta)\int_{s}^{t}e^{-\beta\omega_r}\|A(z_r^{-1}(\omega)\tilde{u}(r,\o,u_0))\|_{V^{*}}\|v\|_{V}dr\nonumber\\
		&\leq C(\omega,\beta,T,\alpha)\int_{s}^{t}(1+\|\tilde{u}(r,\o,u_0))\|_{V}^{\alpha-1})(1+\|\tilde{u}(r,\o,u_0))\|^{\frac{\varpi(\alpha-1)}{\alpha}}_{H})dr\nonumber\\
		&\leq C(\omega,\beta,T,\alpha,\varpi)\left(\int_{s}^{t}(1+\|\tilde{u}(r,\omega,u_0)\|_{V}^\alpha)dr\right)^{\frac{\alpha-1}{\alpha}}(t-s)^{\frac{1}{\alpha}}\\
		&\quad\quad\times(1+\sup_{t\in[0,T]}\|\tilde{u}(t,\o,u_0)\|^{\frac{\varpi(\alpha-1)}{\alpha}}_{H})\nonumber\\
		&\leq C(\omega,\beta,T,\alpha,\varpi)(1+\|\tilde{u}(r,\omega,u_0)\|_{L^{\alpha}([0,T];V)})^{\alpha-1}\nonumber\\
		&\quad\quad\times(1+\sup_{t\in[0,T]}\|\tilde{u}(t,\o,u_0)\|^{\frac{\varpi(\alpha-1)}{\alpha}}_{H})(t-s)^{\frac{1}{\alpha}}.\nonumber
	\end{align}
	Combining with \eqref{eqn:holder-e}-\eqref{eqn:holder-e-2}, we know that $\<u(\cdot,\omega,u_0),v\>_{H}$ has $\gamma_2$-H\"{o}lder regularity for any fixed $v\in
V,\omega\in\Omega$.
	
$\bm\o\in \mathcal{C}^{\gamma}([0,T];\mathbb{R}),\gamma\in(\frac{1}{3},H)\cap(\frac{1}{3},\frac{1}{\alpha}]$. Similar to the Young case, it is clearly that
$\<u(\cdot,\omega,u_0),v\>_{H}\in C^{\gamma}([0,T];\mathbb{R})$. In addition, the remainder term $R^{\<u(\cdot,\omega,u_0),v\>_{H}}$ is given by
	The second equivalence is directly from the previous analysis.
\end{proof}
For the nonautonomous case, we have the same result.
\begin{lem}\label{equvilence1}
	Let $\tilde{u}(t)=z_t(\omega)u(t)$, $u_\tau\in H,\tau\in\mathbb{R}$. Under the assumptions $(A1^\prime)$-$(A4^\prime)$, the following statements are equivalent.
	\begin{description}
		\item[(i)] $\tilde{u}\in C([\tau,\tau+T];H)\cap L^{\alpha}([\tau,\tau+T];V)$ is the variational solution of \eqref{random PDE1} for each fixed $\o\in\Omega$, and
$\<\tilde{u},v\>\in C^{\gamma_2}([\tau,\tau+T];\mathbb{R})$ for any $v\in V$ and $\omega\in \Omega$, where $\gamma_1+\gamma_2>1,\gamma_2\leq \frac{1}{\alpha}\wedge \gamma_1$.
		\item [(ii)] $u\in C([\tau,\tau+T];H)\cap L^{\alpha}([\tau,\tau+T];V)$ and $\<u,v\>_H\in C^{\gamma_2}([\tau,\tau+T];\mathbb{R})$ for $v\in V$ is the variational solution of
\eqref{eqn:spde3} for each fixed $\omega\in\Omega$, where $\gamma_1+\gamma_2>1$.
	\end{description}
\end{lem}
The proof this lemma is same as Lemma \ref{equvilence}, so we omit its proof.
\section{Existence and uniqueness of the stochastic PDEs}\label{well-posedness}
In order to obtain the global solution of random PDEs \eqref{random PDE} and \eqref{random PDE1}, we need to introduce the locally monotone framework for deterministic PDEs(see Refs.
\cite{MR2833734,MR2990049,MR3410409}).
Let $V\subseteq H\subseteq V^{*}$ be a Gelfand triple as before. Consider the following general nonlinear PDEs
\begin{align}
	u'(t) & =A(t,u(t)),\quad  0<t<T,\label{eq:app_e_e}\\
	u(0) & =u_{0}\in H,\nonumber
\end{align}
where $T>0$, $u'$ is the generalized derivative of $u$ on $[0,T]$ and $A:[0,T]\times V\rightarrow V^{*}$ is measurable, $i.e.$ for each
$u\in L^{1}([0,T];V)$, the mapping $t\mapsto A(\cdot,u(\cdot))$ is $V^{*}$-measurable on $[0,T]$.

Assume that for some $\alpha>1,\varpi\ge0$ there exist constants $c,C \ge 0$ and functions $f,g\in L^{1}([0,T];\mathbb{R})$ such that the following
conditions hold for all $t\in[0,T]$ and $v,v_{1},v_{2}\in V$:
\begin{enumerate}
	\item [$(H1)$] (Hemicontinuity) The map $s\mapsto{}_{V^{*}}\<A(t,v_{1}+sv_{2}),v\>_{V}$ is continuous on $\mathbb{R}$.
	\item [$(H2)$] (Local monotonicity)
	\[
	2{}_{V^{*}}\<A(t,v_{1})-A(t,v_{2}),v_{1}-v_{2}\>_{V}\le\left(f(t)+\eta(t,v_{1})+\rho(t,v_{2})\right)\|v_{1}-v_{2}\|_{H}^{2},
	\]
	where $\eta,\rho:[0,T]\times V\rightarrow[0,+\infty)$ are measurable and locally bounded functions for each fixed $t\in[0,T]$.
	\item [$(H3)$] (Coercivity)
	\[
	2{}_{V^{*}}\<A(t,v),v\>_{V}\le-c\|v\|_{V}^{\alpha}+g(t)\|v\|_{H}^{2}+f(t)
	\]
	\item [$(H4)$] (Growth)
	\[
	\|A(t,v)\|_{V^{*}}^{\frac{\alpha}{\alpha-1}}\le C\bigg(f(t)+\|v\|_{V}^{\alpha}\bigg)\bigg(1+\|v\|_{H}^{\varpi}\bigg).
	\]
\end{enumerate}
\begin{thm}
	\label{thm:var_ex} Suppose that $V\subseteq H$ is compact and $(H1)$--$(H4)$ hold. Then for any $u_{0}\in H$, \eqref{eq:app_e_e} has a solution $u$ on
	$[0,T]$, i.e.
	\[
	u\in L^{\alpha}([0,T];V)\cap C([0,T];H),\ u'\in L^{\frac{\alpha}{\alpha-1}}([0,T];V^{*})
	\]
	and
	\[
	\<u(t),v\>_{H}=\<u_{0},v\>_{H}+\int_{0}^{t}{}_{V^{*}}\<A(s,u(s)),v\>_{V}ds,\quad\forall t\in[0,T],v\in V.
	\]
	Moreover, if there exist non-negative constants $C$, $\vartheta$ such that
	\begin{equation}
		\eta(t,v)+\rho(t,v)\le C(1+\|v\|_{V}^{\alpha})(1+\|v\|_{H}^{\vartheta}),\ v\in V,\label{c3}
	\end{equation}
	then the solution of \eqref{eq:app_e_e} is unique.
\end{thm}
The proof of Theorem \ref{thm:var_ex} is same as Theorem 1.1 in  \cite{MR2833734}, thus we omit its proof.
\begin{thm}\label{T2}
	Assume that $V \subseteq H$ is compact and $(H1)$-$(H4)$ hold, $u_i$ are the solution of \eqref{eq:app_e_e}
	with $u_{0}^i\in H$, $i=1,2$, and
	satisfy
	$$ \int_0^T\left( \rho(s,u_1(s))+\eta(s,u_2(s)) \right)  d s<\infty.  $$
	Then
	\begin{equation}
		\begin{split}
			\|u_1(t)-u_2(t)\|_H^2
			\le &   \exp\left[\int_0^t \left(f(t)+\rho(s,u_1(s))+\eta(s,u_2(s)) \right) d s  \right]\\
			& \ \  \cdot \left( \|u_{1,0}-u_{2,0}\|_H^2 \right), \ t\in[0, T].
		\end{split}
	\end{equation}
\end{thm}
Its proof  can be completed  by equation \eqref{eq:app_e_e}, assumption $(H2)$ and Gr\"{o}nwall's inequality(or see Theorem 1.2 in \cite{MR2833734}).
\begin{rem}
	In  this paper, we slightly modify the  condition $(H2)$ in the locally monotone framework for Refs. \cite{MR2833734,MR2990049,MR3410409}. Indeed, the functions $\rho$ and $\eta$
are independent of  the time variable in Refs. \cite{MR2833734,MR2990049,MR3410409}. However, due to the transformation in Section \ref{sec:setup}, we have to consider the functions
$\rho(t,\cdot)$ and $\eta(t,\cdot)$. Although  the functions $\rho$ and $\eta$ depend on the time variable in this paper, the proof of the existence and uniqueness of the solution is
same as Refs. \cite{MR2833734,MR2990049,MR3410409}. In fact,  when we use the condition $(H2)$  to check that the operator $A$ is a pseudo-monotone operator, we usually  fix the time
variable(see Ref. \cite[Lemma 2.2]{MR2833734}), so the time dependence of the function $\rho$ and $\eta$ do not change the method which we prove the well-posedness of the  equation
\eqref{eq:app_e_e} as in Ref. \cite{MR2833734}.
\end{rem}
\begin{rem}
	Recently, R{\"o}ckner et al in Ref. \cite{rockner2022well} gave the  fully local monotone framework, the assumptions for operators $A$ in this paper contain  fully local monotone
coefficient, but our diffusion term is a simple form(linear case). For the nonlinear multiplicative fBm  in the fully local monotone framework, it is a challenging question.
\end{rem}
\begin{thm}\label{wellposed-rpde}
	Assume that $V\subseteq H$ is compact, under the assumptions $(A1)$-$(A4)$ and $u_0\in H$, for each $\omega\in\Omega$, random PDE \eqref{random PDE} has a variational solution
$\tilde{u}(\cdot,\omega,u_0)\in C([0,T],H)\cap L^{\alpha}([0,T],V)$  and
	if there exist non-negative constants $C$, $\vartheta$ such that
	\begin{equation}
		\eta(v)+\rho(v)\le C(1+\|v\|_{V}^{\alpha})(1+\|v\|_{H}^{\vartheta}),\ v\in V,\label{c3*},
	\end{equation}	
	then the solution of \eqref{random PDE} is unique.
\end{thm}
\begin{proof}
	In order to prove the well-posedness of the random PDEs \eqref{random PDE}, we check the conditions $(H1)$-$(H4)$ step by step.
	
	\textit{\textbf{Check (H1):}} Let $\tilde{A}(t,v):=z_t(\omega)A(z^{-1}_{t}(\omega)v),v\in V$. For any $v,v_1,v_2\in V$  and $t\in\mathbb{R}$, then
$z^{-1}_{t}(\omega)v_1,z^{-1}_{t}(\omega)v_2\in V, z_t(\omega)v\in V$. So the map $s\mapsto {}_{V^*}\<\tilde{A}(t,v_1+sv_2),v\>_{V}$ is continuous on $\mathbb{R}$ by $(A1)$.
	
	\textit{\textbf{Check (H2):}} By $(A2)$, we have
	$$2_{V^*}\<\tilde A(t,v_1)-\tilde A(t,v_2),v_1-v_2\>_{V}\le\left(C+\eta(z^{-1}_t(\omega)v_{1})+\rho(z^{-1}_t(\omega)v_{2})\right)\|v_{1}-v_{2}\|_{H}^{2}.$$
	Thus, the operator $\tilde{A}$ satisfies local monotonicity (H2).
	
	\textit{\textbf{Check (H3):}} By $(A3)$ and the continuity of $\omega$, we have
	\begin{equation}\nonumber
		\begin{aligned}
			2{}_{V^{*}}\<\tilde{A}(t,v),v\>_{V}\le -\g e^{(\alpha-2)\beta\omega_t}\|v\|^{\alpha}_{V}+K\|v\|^2_{H}+Ce^{-2\beta\omega_t}.
		\end{aligned}
	\end{equation}
	Therefore, let $c=\min_{t\in[0,T]}\g e^{(\alpha-2)\beta\omega_t}, g(t)=K, f(t)= Ce^{-2\beta\omega_t}$, then $(H3)$ holds.
	
	\textit{\textbf{Check (H4):}} By $(A4)$,  we have
	\begin{equation}\nonumber
		\begin{aligned}
			\|\tilde{A}(t,v)\|_{V^*}^{\frac{\alpha}{\alpha-1}}&\leq
Ce^{-\frac{\beta\alpha}{\alpha-1}\omega_t}(1+e^{\alpha\beta\omega_t}\|v\|_{V}^{\alpha})(1+e^{\beta\varpi\omega_t}\|v\|_{H}^{\varpi})\\
			&\leq C(\omega,T,\varpi,\alpha,\beta)(1+\|v\|_{V}^{\alpha})(1+\|v\|_{H}^{\varpi})
		\end{aligned}
	\end{equation}
	In addition, due to $\eta(t,v)=\eta(z^{-1}_t(\omega)v),\rho(t,v)=\eta(z^{-1}_t(\omega)v)$, and the condition  \eqref{c3*} shows that there exists a constant
$C(\alpha,\vartheta,\varpi,T,\omega)$ such that \eqref{c3} holds. Therefore, Theorem \ref{thm:var_ex}
can be applied, we complete the proof.
\end{proof}
\begin{thm}\label{wellposed-rpde1}
	Assume that $V\subseteq H$ is compact, under the assumptions $(A1^\prime)$-$(A4^\prime)$ and $u_\tau\in H,\tau\in\mathbb{R}$, for each $\omega\in\Omega$, random PDE \eqref{random
PDE1} has a variational solution  $\tilde{u}(\cdot,\tau,\omega,u_0)\in C([\tau,\tau+T],H)\cap L^{\alpha}([\tau,\tau+T],V)$  and
	if there exist non-negative constants $C$, $\vartheta$ such that
	\begin{equation}
		\eta(v)+\rho(v)\le C(1+\|v\|_{V}^{\alpha})(1+\|v\|_{H}^{\vartheta}),\ v\in V,\label{c3*-1}
	\end{equation}	
	then the solution of \eqref{random PDE1} is unique.
\end{thm}
\begin{proof}
	Similarly, for the well-posedness of the random PDEs \eqref{random PDE1}, we  also check the conditions $(H1)$-$(H4)$ step by step.
	
	\textit{\textbf{Check (H1):}} Let $\tilde{A}(t,v):=z_t(\omega)A(t,z^{-1}_{t}(\omega)v),v\in V$. For any $v,v_1,v_2\in V$  and $t\in\mathbb{R}$, then
$z^{-1}_{t}(\omega)v_1,z^{-1}_{t}(\omega)v_2\in V, z_t(\omega)v\in V$. So the map $s\mapsto {}_{V^*}\<\tilde{A}(t,v_1+sv_2),v\>_{V}$ is continuous on $\mathbb{R}$ by $(A1^\prime)$.
	
	\textit{\textbf{Check (H2):}} By $(A2^\prime)$, we have
	$$2_{V^*}\<\tilde A(t,v_1)-\tilde A(t,v_2),v_1-v_2\>_{V}\le\left(f(t)+\eta(z^{-1}_t(\omega)v_{1})+\rho(z^{-1}_t(\omega)v_{2})\right)\|v_{1}-v_{2}\|_{H}^{2}.$$
	Thus, the operator $\tilde{A}$ satisfies local monotonicity (H2).
	
	\textit{\textbf{Check (H3):}} By $(A3^\prime)$ and the continuity of $\omega$, we have
	\begin{equation}\nonumber
		\begin{aligned}
			2{}_{V^{*}}\<\tilde{A}(t,v),v\>_{V}\le -c e^{(\alpha-2)\omega_t}\|v\|^{\alpha}_{V}+g(t)\|v\|^2_{H}+f(t)e^{-2\beta\omega_t}.
		\end{aligned}
	\end{equation}
	Therefore, let $\tilde c=\min_{t\in[\tau,\tau+T]}c e^{(\alpha-2)\beta\omega_t}, \tilde g(t)=g(t), \tilde f(t)= f(t)e^{-2\beta\omega_t}$, then $(H3)$ holds.
	
	\textit{\textbf{Check (H4):}} By $(A4^\prime)$,  we have
	\begin{equation}\nonumber
		\begin{aligned}
			\|\tilde{A}(t,v)\|_{V^*}^{\frac{\alpha}{\alpha-1}}&\leq
Ce^{-\frac{\beta\alpha}{\alpha-1}\omega_t}(f(t)+e^{\alpha\beta\omega_t}\|v\|_{V}^{\alpha})(1+e^{\beta\varpi\omega_t}\|v\|_{H}^{\varpi})\\
			&\leq C(\omega,T,\varpi,\alpha,\beta)(f(t)+\|v\|_{V}^{\alpha})(1+\|v\|_{H}^{\varpi})
		\end{aligned}
	\end{equation}
	In addition, due to $\eta(t,v)=\eta(z^{-1}_t(\omega)v),\rho(t,v)=\eta(z^{-1}_t(\omega)v)$,  the condition  \eqref{c3*-1} shows that there exists a constant
$C(\alpha,\vartheta,\beta,T,\omega)$ such that \eqref{c3} holds. Therefore, Theorem \ref{thm:var_ex} can be applied, we complete the proof.
\end{proof}
According to Lemma \ref{equvilence}, we have the following well-posedness of the stochastic partial differential equations  \eqref{eqn:spde2}.
\begin{cor}\label{well-posedness-o}
	Under the assumptions in Theorem \ref{wellposed-rpde}, stochastic partial differential equations \eqref{eqn:spde2} have a  solution $u\in C([0,T],H)\cap L^{\alpha}([0,T],V)$ and
$\<u,v\>\in C^{\gamma_2}([0,T];\mathbb{R})$ for $\omega\in\Omega$, where $\gamma_1+\gamma_2>1$ and $\gamma_2\leq \frac{1}{\alpha}\wedge \gamma_1$.
\end{cor}
For the well-posdness of stochastic partial differential equations \eqref{eqn:spde3}, we have
\begin{cor}\label{well-posedness-o-1}
	Under the assumptions in Theorem \ref{wellposed-rpde1}, stochastic partial differential equations \eqref{eqn:spde3} have a solution $u\in C([0,T],H)\cap
L^{\alpha}([\tau,\tau+T],V)$ and $\<u,v\>\in C^{\gamma_2}([\tau,\tau+T];\mathbb{R})$ for $\omega\in\Omega$, where $\gamma_1+\gamma_2>1$ and $\gamma_2\leq \frac{1}{\alpha}\wedge
\gamma_1$.
\end{cor}
For the autonomous system \eqref{eqn:spde2} and the nonautonomous  system \eqref{eqn:spde3}, based on Theorem \ref{T2}, we have the following  continuous dependence for  the initial
data.
\begin{thm}\label{dependece-1}
	Assume that $V\subseteq H$ is compact, under the assumptions $(A1)$-$(A4)$, for each $\omega\in\Omega$, $\tilde{u}_i$ are the solution of \eqref{random PDE}
	with $u_{0}^i\in H$, $i=1,2$, and
	satisfy
	$$ \int_0^T\left( \rho(z_s^{-1}(\omega)\tilde u_1(s))+\eta(z_s^{-1}(\omega)\tilde u_2(s)) \right)  d s<\infty. $$
	Then there exists a constant $C$ such that
	\begin{equation}
		\begin{split}
			\|\tilde u_1(t)-\tilde u_2(t)\|_H^2
			\le &   \exp\left[\int_0^t \left(C+\rho(z_s^{-1}(\omega)\tilde u_1(s))+\eta(z_s^{-1}(\omega)\tilde u_2(s)) \right) d s  \right]\\
			& \ \  \cdot \left( \|u_{1,0}-u_{2,0}\|_H^2 \right), \ t\in[0, T].
		\end{split}
	\end{equation}	
	In addition, the solution of  stochastic partial differential equations \eqref{eqn:spde2} is continuously dependent on $u_0$, i.e.
	\begin{equation}
		\begin{split}
			\|u_1(t)- u_2(t)\|_H^2
			\le &   C(T,\omega,\beta)\exp\left[\int_0^t \left(C+\rho(u_1(s))+\eta( u_2(s)) \right) d s  \right]\\
			& \ \  \cdot \left( \|u_{1,0}-u_{2,0}\|_H^2 \right), \ t\in[0, T],
		\end{split}
	\end{equation}	
	where $u_i(t)=z_t^{-1}(\omega)\tilde{u}_i(t)$ are the solution of \eqref{eqn:spde2}
	with $u_{0}^i\in H$, $i=1,2$.
\end{thm}
\begin{thm}\label{dependece-2}
	Assume that $V\subseteq H$ is compact, under the assumptions $(A1^\prime)$-$(A4^\prime)$, for each $\omega\in\Omega$, $\tilde{u}_i$ are the solution of \eqref{random PDE1}
	with $u_{\tau}^i\in H$, $i=1,2$, and
	satisfy
	$$ \int_\tau^{T+\tau}\left( \rho(z_s^{-1}(\omega)\tilde u_1(s))+\eta(z_s^{-1}(\omega)\tilde u_2(s)) \right)  d s<\infty. $$
	Then
	\begin{equation}
		\begin{split}
			\|\tilde u_1(t)-\tilde u_2(t)\|_H^2
			\le &   \exp\left[\int_\tau^t \left(\!f(t)\!+\!\rho(z_s^{-1}(\omega)\tilde u_1(s))+\eta(z_s^{-1}(\omega)\tilde u_2(s)) \right) d s  \right]\\
			& \ \  \cdot \left( \|\tilde u_{1,\tau}-\tilde u_{2,\tau}\|_H^2 \right), \ t\in[\tau, \tau+T].
		\end{split}
	\end{equation}	
	In addition, the solution of  stochastic partial differential equations \eqref{eqn:spde3} is continuously dependent on $u_\tau$, i.e.
	\begin{equation}
		\begin{split}
			\|u_1(t)- u_2(t)\|_H^2
			\le &   C(T,\tau,\omega,\beta)\exp\left[\int_\tau^t \left(f(t)+\rho(u_1(s))+\eta(u_2(s)) \right) d s  \right]\\
			& \ \  \cdot \left( \| u_{1,\tau}- u_{2,\tau}\|_H^2 \right), \ t\in[\tau, \tau+T],
		\end{split}
	\end{equation}	
	where $u_i(t)=z_t^{-1}(\omega)\tilde{u}_i(t)$ are the solution of \eqref{eqn:spde3}
	with $u_{\tau}^i\in H$, $i=1,2$.
\end{thm}
\section{Random attractors\label{sec:construction}}
In this section, we shall consider the random attractors of the system \eqref{eqn:spde2} and \eqref{eqn:spde3}. To this end, we first illustrate that its global solutions in the
sense of Definition  \ref{def:soln_pathw} can generate a random dynamical system or continuous cocycle. Firstly, we give the basic description of the fBm, i.e. the metric dynamical
system.

First of all, let $\Omega=C_0(\mathbb{R},\mathbb{R})$, namely the set of continuous functions and its value are zero at zero, the compact open topology is equipped for $\Omega$.
$\mathcal{F}$ is a Borel algebra which is generated by $\Omega$. $\mathbb{P}$ is the law of the fractional Brownian motion. $\theta$ is a Wiener shift. So
$(\Omega,\mathcal{F},\mathbb{P},\{\theta_{t}\}_{t\in\mathbb{R}})$ is  an ergodic metric dynamical system(see Ref. \cite{MR2836654}). Secondly,  there exists a full measure set
$\Omega_1\subset\Omega$ such that the fractional Brownian motion has a version, denoted by $\omega$, which is $\gamma_1$-H\"{o}lder continuous on any interval $[-k,
k],k\in\mathbb{N}$(see Ref. \cite{MR1070361}), where $\frac{1}{2}<\gamma_1<H$, $\gamma_1+\gamma_2>1$ and $\gamma_2$ is mentioned in Corollary \ref{well-posedness-o} and Corollary
\ref{well-posedness-o-1}.
It is not diffcult to prove  that $\Omega_1$ is $\theta_{t}$-invariant(see Ref. \cite{MR3072986}). Finally, according to the law of iterated logarithm of the fractional Brownian
motion(see Ref. \cite{MR2387368}), there exists a $\theta_t$-invariant subset $\Omega_2\subset\Omega$ with full measure such that for all $\omega\in \Omega_2$,
\begin{equation}\label{sublinear g}
	\lim_{t\rightarrow \pm\infty}\frac{|\omega_t|}{|t|}=0.
\end{equation}
Thus, let $\tilde\Omega=\Omega_1\cap\Omega_2$ and $\tilde{\mathcal{F}}=\mathcal{F}\cap(\Omega_1\cap\Omega_2)$, $\tilde {\mathbb{P}}$ is  restriction of $\mathbb{P}$ on
$\tilde{\mathcal{F}}$. We still use $(\Omega,\mathcal{F},\mathbb{P},\{\theta_{t}\}_{t\in\mathbb{R}})$ represents
$(\tilde\Omega,\tilde{\mathcal{F}},\tilde{\mathbb{P}},\{\theta_{t}\}_{t\in\mathbb{R}})$, it is also an ergodic metric dynamical system.
\begin{rem}\label{cocy-int}
	Due to the definition of Young integral(see Theorem \ref{Young est}), Young integral has the following property
	$$\int_{t_1}^{t_2}Y_rd\omega_r=\int_{t_1-s}^{t_2-s}Y_{r+s}d\theta_{s}\omega_r,$$
	where $t_1,t_2,s\in\mathbb{R}$, $Y_r$ is a $\gamma_2$-H\"{o}lder continuous in Banach space $V^*$ and  $\omega\in \Omega$.
\end{rem}
\begin{thm}\label{rds}
	The solution $u(t,\omega,u_0)$ of the stochastic partial differential equations \eqref{eqn:spde2} can generate a random dynamical system $\varphi:\mathbb{R}^+\times\Omega\times
H\rightarrow H$ as follows
	\begin{equation}\nonumber
		\begin{aligned}
			\varphi(t,\omega,u_0)=u_0+\int_{0}^{t}A(u_r)dr+\int_{0}^{t}\beta u_rd\o_r=u(t,\omega,u_0).
		\end{aligned}
	\end{equation}
	The solution $u(t,\tau,\omega,u_\tau)$ of the stochastic partial differential equations \eqref{eqn:spde3} can generate a continuous cocycle
$\Phi:\mathbb{R}^+\times\mathbb{R}\times\Omega\times H\rightarrow H$ as follows
	\begin{equation}\nonumber
		\begin{aligned}
			\Phi(t,\tau,\omega,u_0)&=u_\tau+\int_{\tau}^{t+\tau}A(r,u_r)dr+\int_{\tau}^{t+\tau}\beta u_rd\theta_{-\tau}\o_r\\
			&=u(t+\tau,\tau,\theta_{-\tau}\omega,u_\tau).
		\end{aligned}
	\end{equation}
\end{thm}
\begin{proof}
	It is clear that $\varphi(0,\omega,u_0)=u_0$. Next, we check the cocycle property. For any $s,t\in\mathbb{R}^{+}$, according to Remark \ref{cocy-int}, we have
	\begin{equation}\nonumber
		\begin{aligned}
			\varphi(t+s,\omega,u_0)&=u_0+\int_{0}^{t+s}A(u_r)dr+\int_{0}^{t+s}\beta u_rd\o_r\\
			&=u_0+\int_{0}^{t}A(u_r)dr+\int_{0}^{t}\beta u_rd\o_r\\
			&\quad\quad+\int_{t}^{t+s}A(u_r)dr+\int_{t}^{t+s}\beta u_rd\o_r\\
			&=\varphi(t,\omega,u_0)+\int_{0}^{s}A(u_{r+t})dr+\int_{0}^{s}\beta u_{r+t}d\theta_{t}\omega_r,
		\end{aligned}
	\end{equation}
	Note that this identity holds on $V^*$. Using the uniqueness of the solution of  the equation \eqref{eqn:spde2}, the above identity shows that
$\varphi(t+s,\omega,u_0)=\varphi(s,\theta_{t}\omega,\varphi(t,\omega,u_0))$. Finally, we begin to check the measurability of $\varphi:\mathbb{R}^+\times\Omega\times H\mapsto H$. Due
to  Corollary \ref{well-posedness-o} and Theorem \ref{dependece-1}, the maps $t\mapsto \varphi(t,\omega,u_0)$ and  $u_0\mapsto \varphi(t,\omega,u_0)$ are continuous. The
measurability  of  the map $\omega\mapsto\varphi(t,\omega,u_0)$ is classical. Indeed, since $\varphi(t,\omega,u_0)=u(t,\omega,u_0)=e^{\beta\omega_t}v(t,\omega,u_0)$, it suffices to
prove the  measurability of the map $\omega\mapsto v(t,\omega,u_0)$, its proof is same as Theorem 1.4 in Ref. \cite{MR2812588}, so we omit its proof. Therefore, we know that mapping
$\varphi:\mathcal{B}(\mathbb{R}^+)\times\mathcal{F}\times\mathbb{R}^m\mapsto \mathbb{R}^m$ is measurable by Ch.3 in Ref. \cite{MR0467310}.  Then the solution of equation
\eqref{eqn:spde2} can generate a continuous random dynamical system $\varphi$. The nonautonomous case can be proved similarly, so we omit its proof.
\end{proof}
\subsection{Autonomous case}
Now, we can consider its random attractor with respect to random dynamical system $\varphi$. The first step to construct random attractors is to find a bounded absorption in $H$. To
this end, we first give  a prior estimate for the solution $u(t,\omega)$.
\begin{lem}
	\label{prop:bdd_ab} Assume that $(A1)$-$(A4)$ and \eqref{c3*} hold. The embedding $V\subseteq H$  is compact. If $\a=2$, additionally assume $K< \gamma\lambda$ in
	$(A3)$. Then there is a random bounded absorbing set $\{F(\o)\}_{\o\in\Omega}$ for $\varphi(t,\omega,\cdot)$.
	More precisely, there exits a measurable function $R:\Omega\to\R_{+}\setminus\{0\}$ such that for all $D\in\mcD$ there is an absorption time $T_{0}=T_{0}(D;\o)$
	such that
	\begin{equation}
		\varphi(t,\omega, D(\theta_{-t}\o))\subseteq B(0,R(\o)),\quad\forall t\ge T_{0},\o\in\Omega.\label{eqn:abs_by_ball}
	\end{equation}
\end{lem}
\begin{proof}
	According to the  equation \eqref{random PDE} and $(A3)$, we have
	\begin{equation}\label{e-s}
		\begin{aligned}
			\frac{d \|\tilde{u}_t\|^2_{H}}{dt}&=2{}_{V^*}\<e^{-\beta\omega_t}A(e^{\beta\omega_{t}}\tilde{u}_t),\tilde{u}_t\>_{V}\\
			&\leq e^{-2\beta\omega_{t}}(-\gamma\|\tilde{u}_te^{\beta\omega_t}\|_{V}^{\alpha}+K\|\tilde{u}_te^{\beta\omega_t}\|_{H}^{2}+C)\\
			&= -\gamma e^{\beta(\alpha-2)\omega_t}\|\tilde{u}_t\|_{V}^{\alpha}+K\|\tilde{u}_t\|_{H}^{2}+Ce^{-2\beta\omega_t}.
		\end{aligned}
	\end{equation}
	In particular, if $\alpha=2$, we have
	\begin{equation}\label{e-s-1}
		\begin{aligned}
			\frac{d \|\tilde{u}_t\|^2_{H}}{dt}&\leq -\gamma \|\tilde{u}_t\|_{V}^{2}+K\|\tilde{u}_t\|_{H}^{2}+Ce^{-2\beta\omega_t}\\
			&\leq -\gamma\lambda \|\tilde{u}_t\|_H^2+K\|\tilde{u}_t\|_{H}^{2}+Ce^{-2\beta\omega_t},
		\end{aligned}
	\end{equation}
	where $K<\lambda\gamma$. For  the case $\alpha>2$, due to
	\begin{equation}\nonumber
		\begin{aligned}
			\lambda\|\tilde{u}_t\|_{H}^2\leq \|\tilde{u}_t\|_{V}^2\leq \|\tilde{u}_t\|_{V}^{\alpha}e^{\beta(\alpha-2)\omega_t}+C(\alpha)e^{-2\beta\omega_t},
		\end{aligned}
	\end{equation}
	then using $(A3)$ and Young's inequality, we have
	\begin{equation}\label{e-s-2}
		\begin{aligned}
			&\frac{d	 \|\tilde{u}_t\|^2_{H}}{dt}=2{}_{V^*}\<e^{-\beta\omega_t}A(e^{\beta\omega_{t}}\tilde{u}_t),\tilde{u}_t\>_{V}\\
			&\leq -\gamma e^{\beta(\alpha-2)\omega_t}\|\tilde{u}_t\|_{V}^{\alpha}+K\|\tilde{u}_t\|_{H}^{2}+Ce^{-2\beta\omega_t}\\
			&\leq -\gamma e^{\beta(\alpha-2)\omega_t} \|\tilde{u}_t\|_V^{\alpha}+\frac{K}{\lambda}\|\tilde{u}_t\|_{V}^{2}+Ce^{-2\beta\omega_t}\\
			&\leq 		 -\gamma e^{\beta(\alpha-2)\omega_t} \|\tilde{u}_t\|_V^{\alpha}\!+\!\epsilon e^{\beta(\alpha-2)\omega_t}
\|\tilde{u}_t\|_V^{\alpha}\!+\!C(\!K,\!\lambda,\alpha, \!\epsilon)e^{-2\beta\omega_t}\\
			&\leq -\lambda(\gamma-\epsilon)\|\tilde{u}_t\|_{H}^2+C(K,\lambda,\gamma,\alpha,\epsilon)e^{-2\beta\omega_{t}},
		\end{aligned}
	\end{equation}
	where $\epsilon<\gamma$. Therefore, for case $\alpha=2$, using Gr\"{o}nwall's inequality, we have
	\begin{equation}\nonumber
		\begin{aligned}
			\|\tilde{u}_t\|_{H}^2\leq e^{-(\lambda\gamma-K)t}\|u_0\|_{H}^{2}+C\int_{0}^{t}e^{-2\beta\omega_r+(\lambda\gamma-K)(r-t)}dr,
		\end{aligned}
	\end{equation}
	it means that
	\begin{equation}\nonumber
		\begin{aligned}
			\|u_t\|_{H}^2&\leq e^{-(\lambda\gamma-K)t+2\beta\omega_t}\|u_0\|_{H}^{2}+C\int_{0}^{t}e^{-2\beta(\omega_r-\omega_t)-(\lambda\gamma-K)(t-r)}dr\\
			&=e^{-(\lambda\gamma-K)t+2\beta\omega_t}\|u_0\|_{H}^{2}+C\int_{-t}^{0}e^{-2\beta(\theta_t\omega_{r})+(\lambda\gamma-K)r}dr.
		\end{aligned}
	\end{equation}
	Then, for $u_0\in D(\theta_{-t}\omega)\in\mathcal{D},t>0$,
	\begin{equation}\label{Bounde set1}
		\begin{aligned}
			\|&\varphi(t,\theta_{-t}\omega,u_0)\|^2_H=\|u(t,\theta_{-t}\omega,u_0)\|_{H}^{2}\\
			&\leq e^{-(\lambda\gamma-K)t-2\beta\omega_{-t}}\|u_0\|_{H}^{2}+C\int_{-t}^{0}e^{-2\beta\omega_{r}+(\lambda\gamma-K)r}dr.\\
		\end{aligned}
	\end{equation}
	For the first term on the right hand  of \eqref{Bounde set1},  due to the sublinear growth \eqref{sublinear g} of $\omega$ and the temperedness of $u_0$,  there exists
$T_0(D,\omega)>0$, for $t>T_0(\omega)$, such that $-2\beta\omega_{-t}\leq \frac{\lambda\gamma-K}{4}t,\|u_0\|_{H}^2\leq e^{\frac{\lambda\gamma-K}{4}t}$. So
	\begin{equation}\label{Bounde set1-1}
		\begin{aligned}
			e^{-(\lambda\gamma-K)t-2\beta\omega_{-t}}\|u_0\|_{H}^{2}<1.
		\end{aligned}
	\end{equation}
	For the second term on the right hand  of \eqref{Bounde set1},  due to the sublinear growth \eqref{sublinear g}  of $\omega$, there exists $T_1(\omega)>0$ such that
$-2\beta\omega_r\leq \frac{-(\lambda\gamma-K)r}{2},r\leq -T_1(\omega)$. So we have
	\begin{equation}\label{Bounde set1-2}
		\begin{aligned}
			\int_{-t}^{0}e^{-2\beta\omega_{r}+(\lambda\gamma-K)r}dr&\leq \int_{-\infty}^{0}e^{-2\beta\omega_{r}+(\lambda\gamma-K)r}dr\\
			&=\int_{-T_1(\omega)}^{0}e^{-2\beta\omega_{r}+(\lambda\gamma-K)r}dr+\int_{-\infty}^{-T_1(\omega)}e^{-2\beta\omega_{r}+(\lambda\gamma-K)r}dr\\
			&=\int_{-T_1(\omega)}^{0}e^{-2\beta\omega_{r}+(\lambda\gamma-K)r}dr+\int_{-\infty}^{-T_1(\omega)}e^{\frac{(\lambda\gamma-K)r}{2}}dr<\infty.
		\end{aligned}
	\end{equation}
	By \eqref{Bounde set1}-\eqref{Bounde set1-2}, there exists $\tilde{T}_0(D,\omega)=\max\{T_0(D,\omega),T_1(\omega)\}>0$, such that
	\begin{equation}\label{bounded absorbing set-1}
		\begin{aligned}
			\|\varphi(t,\theta_{-t}\omega,u_0)\|^2_H\leq 1+C\int_{-\infty}^{0}e^{-2\beta\omega_{r}+(\lambda\gamma-K)r}dr=:R^2(\omega),\quad t\geq \tilde{T}_0(\omega).
		\end{aligned}
	\end{equation}
	Similarly, for $\alpha>2$, $u_0\in D(\theta_{-t}\omega)\in\mathcal{D},t>0$, according to \eqref{e-s-2},  there exists $\tilde{\tilde{T}}_0(D,\omega)>0$, such that
	\begin{equation}\label{bounded absorbing set-2}
		\begin{aligned}
			\|\varphi(t,\theta_{-t}\omega,u_0)\|^2_H&\leq\! 1+\\
			&\!C(K,\lambda,\gamma,\alpha,\epsilon)\int_{-\infty}^{0}\!e^{-2\beta\omega_{r}+\lambda(\gamma-\epsilon)r}dr=:R^2(\omega), t\geq \tilde{\tilde{T}}_0(\omega).\\
		\end{aligned}
	\end{equation}
	By \eqref{bounded absorbing set-1} and \eqref{bounded absorbing set-2}, the random dynamical system $\varphi$ has a bounded absorbing set. Note that $R(\theta_t\omega)$ is
tempered. Indeed, for the case $\alpha=2$ and any $\eta>0$, we have
	\begin{align*}
		\lim_{t\rightarrow-\infty}e^{\eta t}R(\theta_{-t}\omega)\leq \lim_{t\rightarrow-\infty}e^{\eta t}R^2(\theta_{-t}\omega)=\lim_{t\rightarrow-\infty}Ce^{\eta
t}\int_{-\infty}^{0}e^{-2\beta\theta_t\omega_{r}+(\lambda\gamma-K)r}dr,
	\end{align*}
	let $\eta^\prime<(\lambda\gamma-K)\wedge \eta$, using the sublinear growth \eqref{sublinear g} of $\o$, there exists a $T_0(\omega)>0$ such that $-2\beta\theta_{t}\omega_r\leq
-\eta^\prime(t+r)$ for $t\leq -T_0(\omega)$ and $r<0$, then
	\begin{align*}
		&\lim_{t\rightarrow-\infty}Ce^{\eta t}\int_{-\infty}^{0}e^{-2\beta\theta_t\omega_{r}+(\lambda\gamma-K)r}dr\\
		&=\lim_{t\rightarrow-\infty}Ce^{\eta t}\int_{-T_0(\omega)}^{0}e^{-2\beta\theta_t\omega_{r}+(\lambda\gamma-K)r}dr+\lim_{t\rightarrow-\infty}Ce^{\eta
t}\int_{-\infty}^{-T_0(\omega)}e^{-2\beta\theta_t\omega_{r}+(\lambda\gamma-K)r}dr\\
		&\leq \lim_{t\rightarrow-\infty}Ce^{\eta t}\int_{-T_0(\omega)}^{0}e^{-2\beta\theta_t\omega_{r}+(\lambda\gamma-K)r}dr+\lim_{t\rightarrow-\infty}Ce^{\eta
t}\int_{-\infty}^{-T_0(\omega)}e^{-\eta^\prime(t+r)+(\lambda\gamma-K)r}dr\\
	\end{align*}
	For the first term  on the right side of the above inequality, $-2\beta\theta_{t}\omega_r\leq -\eta^\prime(t+r)$ for $t\leq -T_0(\omega)$, thus
	$$\lim_{t\rightarrow-\infty}Ce^{\eta t}\int_{-T_0(\omega)}^{0}e^{-2\beta\theta_t\omega_{r}+(\lambda\gamma-K)r}dr=0.$$
	It is clearly that
	$$\lim_{t\rightarrow-\infty}Ce^{\eta t}\int_{-\infty}^{-T_0(\omega)}e^{-\eta^\prime(t+r)+(\lambda\gamma-K)r}dr=0.$$
	That means that $R(\theta_{t}\omega)$ is tempered for case $\alpha=2$, we can get the same result for the case $\alpha>2$.
\end{proof}
\begin{thm}\label{thm:random attractor}
	Assume that $(A1)$-$(A4)$ and $\eqref{c3*}$ hold. The embedding $V\subseteq H$  is compact. If $\a=2$, additionally assume $K< \gamma\lambda$ in
	$(A3)$.  Then there is a unique random attractor $\mcA$ for the random dynamical system $\varphi$.
\end{thm}
\begin{proof}
	In order to construct the random attractor of the random dynamical system $\varphi$, we need to find a compact absorbing set. Due to
$\varphi(t,\theta_{-t}\omega,D(\theta_{-t}\omega))=e^{\beta\theta_{-t}\omega_t}\tilde{u}(t,\theta_{-t}\omega,D(\theta_{-t}\omega))$, where $D\in\mathcal{D}$ and $t\geq 0$. Then we
need to prove $\tilde{u}(t,\omega,B)$ is a  compact set for any bounded set $B\subset H$ and $\o\in\Omega$. Namely $\tilde{u}^n\in \tilde{u}(t,\omega,B),$ i.e.
$\tilde{u}^n=\tilde{u}(t,\omega,b_n)$  for some sequence $b_n\in B\subseteq H$. We need to prove the existence of a convergent subsequence of $\tilde{u}^n$. According to \eqref{e-s},
we have
	\begin{equation}
		\begin{aligned}
			\|\tilde{u}(t,\omega,b_n)\|^2_{H}&\leq\|b_n\|_{H}^2-\gamma\int_{0}^{t} e^{\beta(\alpha-2)\omega_r}\|\tilde{u}(r,\omega,b_n)\|_{V}^{\alpha}dr\\
			&\quad+K\int_{0}^{t}\|\tilde{u}(r,\omega,b_n)\|_{H}^{2}dr+C\int_{0}^{t}e^{-2\beta\omega_r}dr.
		\end{aligned}
	\end{equation}
	Thus, there exists a $C>0$ such that $\int_{0}^{t} \|\tilde{u}(r,\omega,b_n)\|_{V}^{\alpha}dr<C$. Since $\tilde{u}(t,\omega,b_n)$ satisfy the following equation
	$$\tilde{u}(t,\omega,b_n)=b_n+\int_{0}^{t}e^{-\beta\omega_{r}}A(e^{\beta\omega_{r}}\tilde{u}(r,\omega,b_n))dr,t\geq 0.$$
	in $V^*$. Then $\frac{d \tilde{u}(t,\omega,b_n)}{dt}$ exists  in $V^*$ for any $\omega\in\Omega$ by Theorem \ref{thm:var_ex} and satisfies
	$$\frac{d \tilde{u}(t,\omega,b_n)}{dt}=e^{-\beta\omega_{t}}A(e^{\beta\omega_{t}}\tilde{u}(t,\omega,b_n)), \quad t\geq 0.$$
	Thus,
	\begin{equation}
		\begin{aligned}
			\int_{0}^{t}\left\|\frac{d
\tilde{u}(r,\omega,b_n)}{dr}\right\|^{\frac{\alpha}{\alpha-1}}_{V^*}dr&=\int_{0}^{t}\left\|e^{-\beta\omega_{r}}A(e^{\beta\omega_{r}}\tilde{u}(r,\omega,b_n))\right\|_{V^*}^{\frac{\alpha}{\alpha-1}}dr\\
			&\leq C\int_{0}^{t}\bigg[e^{\frac{-\beta\alpha\omega_r}{\alpha-1}}(1+e^{\alpha\beta\omega_r}\|\tilde{u}(r,\omega,b_n)\|_{V}^{\alpha})\\
			&\quad\times(1+e^{\beta\varpi\omega_r}\|\tilde{u}(r,\omega,b_n)\|_{H}^{\varpi})\bigg]dr\\
			&\leq C\int_{0}^{t}1+\|\tilde{u}(r,\omega,b_n)\|_{V}^{\alpha}dr<C,
		\end{aligned}
	\end{equation}
	where we use the estimate for $\|\tilde{u}(r,\omega,b_n)\|_{H}$ in Lemma \ref{prop:bdd_ab}, and it is uniformly bounded for $n$ by the boundedness of the sequence $b_n$. Then the
sequence $\tilde{u}(t,\omega,b_n)$ in the space
	\begin{equation}\nonumber
		\begin{aligned}
			&W:=\left\{v\in L^\alpha([0,t],V),\frac{d}{ds}v\in L^{\frac{\alpha}{\alpha-1}}([0,t],V^{*})\right\},\\
			&\|v\|_{W}=\|v\|_{ L^\alpha([0,t],V)}+\left\|\frac{d}{ds}v\right\|_{L^{\frac{\alpha}{\alpha-1}}([0,t],V^{*})}<\infty.
		\end{aligned}
	\end{equation}
	Since the embedding $W\subseteq L^\alpha([0,t],H)$ is compact(see Theorem 2.1 in Ref. \cite{MR769654}). So $\{\tilde{u}(t,\omega,b_n)\}$ is  compact in $L^{\alpha}([0,t],H)$, and
we can choose a subsequence of $b_n$, we still denote by $b_n$, and $Z_0\in L^{\alpha}([0,t],H)$ such that
	$$\tilde{u}(\cdot,\omega,b_n)\rightarrow Z_0\quad \textbf{in}\quad L^{\alpha}([0,t],H).$$
	Further, we can choose a subsequence of $b_n$(again denote by $b_n$) such that
	$$\tilde{u}(r,\omega,b_n)\rightarrow Z_0(r)\quad \textbf{in}\quad H,$$
	for almost every $r\in[0,t]$. By  the continuous dependence of the initial data,  then we can choose such $r\in [0,t]$ such that
	$$\tilde{u}(t,\omega,b_n)=\tilde{u}(t,0,\omega,b_n)=\tilde{u}(t,r,\omega,\cdot)\circ \tilde{u}(r,0,\omega,b_n)\rightarrow \tilde{u}(t,r,\omega,Z_0).$$
	That is to say, we find a convergent subsequence of $\tilde{u}(r,\omega,b_n)$.  Thus $\tilde{u}$ is a precompact operator on $H$ for any $t\geq 0$ and $\o\in\Omega$, it also
means that $\varphi(t,\omega,\cdot)$ is a compact operator on $H$. Since $F(\theta_{-1}\omega)$ is a bounded set in $H$. Hence, there exists a compact set $K(\omega)$:
	$$K(\omega):=\overline{\varphi(1,\theta_{-1}\omega,F(\theta_{-1}\omega))}.$$
	It is also an absorbing set, for any $D\in\mathcal{D}$, there exists $T_0(\omega,D)>0$ such that
	\begin{equation}\nonumber
		\begin{aligned}
			\varphi(t,\theta_{-t}\omega,D(\theta_{-t}\omega))&=\varphi(1,\theta_{-1}\omega,\varphi(t-1,\theta_{-t}\omega,D(\theta_{-t}\omega)))\\
			&\subseteq\varphi(1,\theta_{-1}\omega,F(\theta_{-1}\omega))\subseteq K(\omega)
		\end{aligned}
	\end{equation}
	holds for $t\geq T_0(D,\omega)$. According to Theorem \ref{attrctor-a},  the random dynamical system $\varphi$ has a unique random attractor $\mathcal{A}(\omega)$,
	$$\mathcal{A}(\omega)=\bigcap_{t\geq 0}\overline{\bigcup_{r\geq t}\varphi(r,\theta_{-r}\omega,K(\omega))}.$$
\end{proof}
\subsection{Nonautonomous case} For the nonautonomous systems \eqref{eqn:spde3}, we have the similar results as follows:
\begin{lem}
	\label{prop:bdd_ab1-n} Assume $(A1^\prime)$-$(A4^\prime)$ with function $g$ which is independent of $t$ and the function $|f|$ is exponential integrability and
$\lim_{t\rightarrow -\infty}e^{-\eta t}\int_{-\infty}^{0}e^{\eta_1 r}|f(r+t)|dr=0$ for any $\eta>0$ and $\eta_1<\lambda(c-\epsilon),\epsilon<c$. The embedding $V\subseteq H$  is
compact and \eqref{c3*-1} holds. If $\a=2$, additionally assume $g< c\lambda$ in
	$(A3^\prime)$. Then there is a random bounded $\mcD$-pullback absorbing set $\{F(\tau,\o)\}_{\tau\in\mathbb{R},\o\in\Omega}$ for $\Phi(t,\tau-t,\theta_{-t}\omega,\cdot)$.
	More precisely, there is a measurable function $R:\mathbb{R}\times\Omega\to\R_{+}\setminus\{0\}$ such that for all $D\in\mcD$ there is an absorption time $T_{0}=T_{0}(D,\tau,\o)$
	such that
	\begin{equation}
		\Phi(t,\tau-t,\theta_{-t}\omega, D(\theta_{-t}\o))\subseteq B(0,R(\tau,\o)),\quad\forall t\ge T_{0},\o\in\Omega.\label{eqn:abs_by_ball-n}
	\end{equation}
\end{lem}
\begin{proof}
	According to the  equation \eqref{random PDE1} and $(A3^\prime)$, we have
	\begin{equation}\label{e-s-n}
		\begin{aligned}
			\frac{d \|\tilde{u}_t\|^2_{H}}{dt}\leq -c e^{\beta(\alpha-2)\omega_t}\|\tilde{u}_t\|_{V}^{\alpha}+g\|\tilde{u}_t\|_{H}^{2}+f(t)e^{-2\beta\omega_t}.
		\end{aligned}
	\end{equation}
	In particular, if $\alpha=2$, we have
	\begin{equation}\label{e-s-1-n}
		\begin{aligned}
			\frac{d \|\tilde{u}_t\|^2_{H}}{dt}\leq -c\lambda \|\tilde{u}_t\|_H^2+g\|\tilde{u}_t\|_{H}^{2}+f(t)e^{-2\beta\omega_t},
		\end{aligned}
	\end{equation}
	where $g<c\lambda$. For  the case $\alpha>2$, due to
	\begin{equation}\nonumber
		\begin{aligned}
			\lambda\|\tilde{u}_t\|_{H}^2\leq \|\tilde{u}_t\|_{V}^2\leq \|\tilde{u}_t\|_{V}^{\alpha}e^{\beta(\alpha-2)\omega_t}+C(\alpha)e^{-2\beta\omega_t},
		\end{aligned}
	\end{equation}
	then using $(A3^\prime)$ and Young's inequality, we have
	\begin{equation}\label{e-s-2-n}
		\begin{aligned}
			&\frac{d	 \|\tilde{u}_t\|^2_{H}}{dt}=2{}_{V^*}\<e^{-\beta\omega_t}A(t,e^{\beta\omega_{t}}\tilde{u}_t),\tilde{u}_t\>_{V}\\
			&\leq -c e^{\beta(\alpha-2)\omega_t}\|\tilde{u}_t\|_{V}^{\alpha}+g\|\tilde{u}_t\|_{H}^{2}+f(t)e^{-2\beta\omega_t}\\
			&\leq -c e^{\beta(\alpha-2)\omega_t} \|\tilde{u}_t\|_V^{\alpha}+\frac{g}{\lambda}\|\tilde{u}_t\|_{V}^{2}+f(t)e^{-2\beta\omega_t}\\
			&\leq 		 -ce^{\beta(\alpha-2)\omega_t} \|\tilde{u}_t\|_V^{\alpha}\!+\!\epsilon e^{\beta(\alpha-2)\omega_t}
\|\tilde{u}_t\|_V^{\alpha}\!+(\!C(\!g,\!\lambda,\alpha,\!\epsilon)+f(t))e^{-2\beta\omega_t}\\
			&\leq -\lambda(c-\epsilon)\|\tilde{u}_t\|_{H}^2+(C(g,\lambda,c,\alpha,\epsilon)+f(t))e^{-2\beta\omega_{t}},
		\end{aligned}
	\end{equation}
	where $\epsilon<c$.	Therefore, for case $\alpha=2$, by Gr\"{o}nwall's inequality, we have
	\begin{equation}\nonumber
		\begin{aligned}
			\|\tilde{u}_t\|_{H}^2\leq e^{-(c\lambda-g)(t-\tau)}\|\tilde u_\tau\|_{H}^{2}+\int_{\tau}^{t}e^{-2\beta\omega_r+(c\lambda-g)(r-t)}f(r)dr,
		\end{aligned}
	\end{equation}
	it means that
	\begin{equation}\nonumber
		\begin{aligned}
			\|u_t\|_{H}^2&\leq e^{-(c\lambda-g)(t-\tau)+2\beta(\omega_t-\omega_\tau)}\|u_\tau\|_{H}^{2}+\int_{\tau}^{t}e^{-2\beta(\omega_r-\omega_t)-(c\lambda-g)(t-r)}f(r)dr\\
			&=e^{-(c\lambda-g)(t-\tau)+2\beta(\omega_t-\omega_\tau)}\|u_\tau\|_{H}^{2}+\int_{\tau-t}^{0}e^{-2\beta(\theta_t\omega_{r})+(c\lambda-g)r}f(r+t)dr.
		\end{aligned}
	\end{equation}
	Then, for $u_{\tau-t}\in D(\tau-t,\theta_{-t}\omega)\in\mathcal{D},t>0$, we have
	\begin{equation}\label{Bounde set1-n}
		\begin{aligned}
			\|&\Phi(t,\tau-t,\theta_{-t}\omega,u_{\tau-t})\|^2_H=\|u(\tau,\tau-t,\theta_{-\tau}\omega,u_{\tau-t})\|_{H}^2\\
			&\leq e^{-(c\lambda-g)t-2\beta\omega_{-t}}\|u_{\tau-t}\|_{H}^{2}+\int_{-t}^{0}e^{-2\beta\omega_{r}+(c\lambda-g)r}f(r+\tau)dr.\\
		\end{aligned}
	\end{equation}
	For the first term on the right hand  of \eqref{Bounde set1-n},  due to the sublinear growth \eqref{sublinear g} of $\omega$ and the temperedness of $u_{\tau-t}$,  there exists
$T_0(D,\omega)>0$, for $t>T_0(\omega)$, such that $-2\beta\omega_{-t}\leq \frac{c\lambda-g}{4}t,\|u_{\tau-t}\|_{H}^2\leq e^{\frac{c\lambda-g}{4}t}$. So
	\begin{equation}\label{Bounde set1-1-n}
		\begin{aligned}
			e^{-(c\lambda-g)t-2\beta\omega_{-t}}\|u_{\tau-t}\|_{H}^{2}<1.
		\end{aligned}
	\end{equation}
	For the second term on the right hand  of \eqref{Bounde set1-n},  due to  the sublinear growth \eqref{sublinear g} of $\omega$, there exists $T_1(\omega)>0$ such that
$-2\beta\omega_r\leq \frac{-(c\lambda-g)r}{2},r\leq -T_1(\omega)$. So we have
	\begin{equation}\label{Bounde set1-2-n}
		\begin{aligned}
			\int_{-t}^{0}&e^{-2\beta\omega_{r}+(c\lambda-g)r}f(r+\tau)dr
			\leq\! \int_{-\infty}^{0}e^{-2\beta\omega_{r}+(c\lambda-g)r}|f(r+\tau)|dr\\
			&\!=\!\int_{-T_1(\omega)}^{0}e^{-2\beta\omega_{r}+(c\lambda-g)r}|f(r+\tau)|dr+\int_{-\infty}^{-T_1(\omega)}e^{-2\beta\omega_{r}+(c\lambda-g)r}|f(r+\tau)|dr\\
			&=\int_{-T_1(\omega)}^{0}e^{-2\beta\omega_{r}+(c\lambda-g)r}|f(r+\tau)|dr+\int_{-\infty}^{-T_1(\omega)}e^{\frac{(c\lambda-g)r}{2}}|f(r+\tau)|dr<\infty,
		\end{aligned}
	\end{equation}
	where we use $|f|$ is exponential integrability to the last integral. By \eqref{Bounde set1-n}-\eqref{Bounde set1-2-n}, there exists
$\tilde{T}_0(D,\omega)=\max\{T_0(D,\omega),T_1(\omega)\}>0$, such that
	\begin{equation}\label{bounded absorbing set-1-n}
		\begin{aligned}
			\|&\Phi(t,\tau-t,\theta_{-t}\omega,u_{\tau-t})\|^2_H\leq 1\\
			&+\int_{-\infty}^{0}e^{-2\beta\omega_{r}+(c\lambda-g)r}|f(r+\tau)|dr=:R^2(\tau,\omega),\quad t\geq \tilde{T}_0(\omega).
		\end{aligned}
	\end{equation}
	Similarly, for $\alpha>2$, $u_{\tau-t}\in D(\tau-t,\theta_{-t}\omega)\in\mathcal{D},t>0$, according to \eqref{e-s-2-n},  there exists $\tilde{\tilde{T}}_0(D,\omega)>0$, such that
	\begin{equation}\label{bounded absorbing set-2-n}
		\begin{aligned}
			\|&\Phi(t,\tau-t,\theta_{-t}\omega,u_{\tau-t})\|^2_H\\&\leq 1+
			\!\int_{-\infty}^{0}(C(g,\lambda,c,\alpha,\epsilon)+|f(r+\tau)|)e^{-2\beta\omega_{r}+\lambda(c-\epsilon)r}dr\\
			&=:R^2(\tau,\omega), t\geq \tilde{\tilde{T}}_0(\omega).\\
		\end{aligned}
	\end{equation}
	By \eqref{bounded absorbing set-1-n} and \eqref{bounded absorbing set-2-n}, the continuous cocycle $\Phi$ has a bounded absorbing set.  Note that $R(\tau+t,\theta_{t}\omega)$ is
tempered. Indeed, without loss of generality, we only discuss the case $\alpha>2$.  For any $\eta>0$, similar to the autonomous case, we have
	\begin{align*}
		\lim_{t\rightarrow -\infty}e^{\eta t}\int_{-\infty}^{0}C(g,\lambda,c,\alpha,\epsilon)e^{-2\beta\theta_{t}\omega_r+\lambda(c-\epsilon)r}dr=0.
	\end{align*}
	Due to the sublinear growth \eqref{sublinear g} of $\omega$, there exists a  $T_0(\omega)>0$ such that $-2\beta\theta_t\omega_r<-\eta^\prime(r+t)$ for any
$\eta^\prime<\eta\wedge(\lambda(c-\epsilon))$, $t<-T_0(\omega)$ and $r<0$, we have
	\begin{align*}
		&\lim_{t\rightarrow -\infty}e^{\eta t}\int_{-\infty}^{0}|f(r+\tau+t)|e^{-2\beta\theta_t\omega_{r}+\lambda(c-\epsilon)r}dr\\
		&=\lim_{t\rightarrow -\infty}e^{\eta t}\int_{-T_0(\omega)}^{0}|f(r+\tau+t)|e^{-2\beta\theta_t\omega_{r}+\lambda(c-\epsilon)r}dr\\
		&~~+\lim_{t\rightarrow -\infty}e^{\eta t}\int_{-\infty}^{-T_0(\omega)}|f(r+\tau+t)|e^{-\eta^\prime(r+t)+\lambda(c-\epsilon)r}dr\\
		&\leq \lim_{t\rightarrow -\infty}e^{\eta t}\int_{-T_0(\omega)}^{0}|f(r+\tau+t)|e^{-2\beta\theta_t\omega_{r}+\lambda(c-\epsilon)r}dr\\
		&~~+e^{-(\eta-\eta^\prime)\tau}\lim_{t\rightarrow -\infty}e^{(\eta-\eta^\prime) t}\int_{-\infty}^{-T_0(\omega)}|f(r+t)|e^{-\eta^\prime r+\lambda(c-\epsilon)r}dr.
	\end{align*}
	According to the assumption of this theorem, the second term on the right side of the above inequality  is equal to zero. For the first term of on the right side of the above
inequality, due to
	$$e^{-(\eta-\eta^\prime)\tau}\lim_{t\rightarrow -\infty}e^{(\eta-\eta^\prime) t}\int_{-\infty}^{0}|f(r+t)|e^{-\eta^\prime r+\lambda(c-\epsilon)r}dr=0,$$
	then, for any $\varepsilon>0$, there exists a $\delta(\varepsilon)<0$ such that
	$$e^{-(\eta-\eta^\prime)\tau}e^{(\eta-\eta^\prime) t}\int_{-\infty}^{0}|f(r+t)|e^{-\eta^\prime r+\lambda(c-\epsilon)r}dr<\varepsilon,\quad \forall t<\delta(\varepsilon).$$
	Thus, let $\delta_1(\varepsilon,\omega)=\min\{-T_0(\omega),\delta(\varepsilon)\}$, we have
	\begin{align*}
		&e^{\eta t}\int_{-T_0(\omega)}^{0}|f(r+\tau+t)|e^{-2\beta\theta_t\omega_{r}+\lambda(c-\epsilon)r}dr\\
		&\leq e^{-(\eta-\eta^\prime)\tau}e^{(\eta-\eta^\prime) t}\int_{-\infty}^{0}|f(r+t)|e^{-\eta^\prime r+\lambda(c-\epsilon)r}dr\leq \varepsilon.  \quad \forall
t<\delta_1(\varepsilon,\omega).
	\end{align*}
	Therefore, we know that
	$$
	\lim_{t\rightarrow -\infty}e^{\eta t}\int_{-\infty}^{0}|f(r+\tau+t)|e^{-2\beta\theta_t\omega_{r}+\lambda(c-\epsilon)r}dr =0.
	$$
	Then $\lim_{t\rightarrow 0}e^{\eta t}R(\tau-t,\theta_{t}\omega)\leq \lim_{t\rightarrow 0}e^{\eta t}R^2(\tau-t,\theta_{t}\omega)=0$,  i.e. $R(\tau-t,\theta_{t}\omega)$ is
tempered.
\end{proof}
\begin{thm}\label{thm:random attractor-n}
	Assume $(A1^\prime)$--$(A4^\prime)$ with function $g$ which is independent of $t$ and $|f|$ is exponential integrability and $\lim_{t\rightarrow -\infty}e^{-\eta
t}\int_{-\infty}^{0}e^{\eta_1 r}|f(r+t)|dr=0$ for any $\eta>0$ and $\eta_1<\lambda(c-\epsilon),\epsilon<c$. The embedding $V\subseteq H$  is compact and \eqref{c3*-1} holds. If
$\a=2$, additionally assume $g< c\lambda$ in
	$(A3^\prime)$. Then there is a  unique $\mcD$-pullback random attractor $\{\mathcal A(\tau,\o)\}_{\tau\in\mathbb{R},\o\in\Omega}$ for the continuous cocycle  $\Phi$.
\end{thm}
\begin{proof}
	its proof is similar to the Theorem \ref{thm:random attractor}, According to Theorem \ref{att} and Remark  \ref{asymptotically compact}, By the existence of the bounded absorbing
set in Lemma \ref{prop:bdd_ab1-n} and  the precompactness of the cocycle $\Phi$, which proof is same as Theorem \ref{thm:random attractor}, then we  complete the proof.
\end{proof}
\section{Example}\label{example}
In this section, we give two examples to illustrate  that how to apply the previous theory to obtain the dynamics. In addition, using our results, some other models in Refs.
\cite{MR3027636,MR2812588,MR4097253}  driven by the linear multiplicative fBm  have random attractor or $\mathcal{D}$-pullback random attractor.
\begin{ex}[Stochastic porous medium equation] Let $\mathcal{O}$ be a bounded open set in $\mathbb{R}^d,d\geq 1$. For any $r>1$, we introduce the following triple
	$$V :=L^{r+1}(\mathcal{O})\subseteq H:=W_{0}^{-1,2}(\mathcal{O})\subseteq V^*$$
	and the stochastic porous media equation
	$$du=\triangle (|u|^{r-1}u)+\beta ud\omega_t,$$
	where $\beta\in\mathbb{R}$. Then it generates a random dynamical system $\varphi$, and  it has a unique random attractor $\mathcal{A}$.
	\begin{proof}
		Firstly, it is easy to check $(A1)$-$(A4)$. Indeed, according to $(H1)$-$(H4)$ in \cite[page 71]{MR2329435},  $(A1)$ holds, $\eta=\rho=0,C=2c|\mathcal{O}|$ for $(A2)$,
$\alpha=r+1,K=0,\gamma=2,C=2c|\mathcal{O}|$ for $(A3)$,  $\varpi=0,C=2c|\mathcal{O}|$ for $(A4)$, where $c>0$ is appropriate constant and $|\mathcal{O}|$ is a Lebesgue measure of
$\mathcal{O}$. Finally, Since $L^2(\mathcal{O})\subseteq H$  is  compact, then so is $L^{r+1}(\mathcal{O})\subseteq H$. Thus by Theorem \ref{rds} and Theorem \ref{thm:random
attractor}, the solution of stochastic porous media equation can generate a random dynamical system $\varphi$, and it has a unique random attractor $\mathcal{A}(\o)$.
	\end{proof}
\end{ex}
\begin{ex}[$2D$-stochastic Navier-Stokes equation]
	Now we consider the $2D$-stochastic Navier-Stokes equation. Let $\mathcal{O}$ is a bounded domain in $\mathbb{R}^2$ with regular boundary $\Gamma$, we introduce the following
spaces
	$$
	\mathcal{V}=\left\{\phi \in\left(\mathcal{C}_0^{\infty}(\mathcal{O})\right)^2: \nabla \cdot \phi=0\right\}
	$$
	
	$H=$ closure of $\mathcal{V}$ in $\left(L^2(\mathcal{O})\right)^2$ with inner product $(\cdot, \cdot)$ and associate norm $|\cdot|$,
	
	$V=$ closure of $\mathcal{V}$ in $\left(H_0^1(\mathcal{O})\right)^2$ with inner product $((\cdot, \cdot))$ and associate norm $\|\cdot\|$,
	
	$H^{*}=$ dual space of $H, V^{*}=$ dual space of $V$ with norm $\|\cdot\|_{V^{*}}$, and ${}_{V^*}\langle\cdot, \cdot\rangle_V$ denotes the dual pairing between $V$ and $V^{*}$.
	
	Thus we have the  Gelfand triple $V\subseteq H=H^*\subseteq V^*$, where the embeddings are dense and compact. More precisely, we consider the following model:
	\begin{equation}
		\left\{\begin{array}{l}
			\frac{\partial u}{\partial t}-\nu \Delta u+(u(t) \cdot \nabla) u+\nabla p=h(t)+\beta u\dot{\omega}_t \text { in } \mathcal{O} \times(\tau, \infty), \\
			\operatorname{div} u=0 \text { in } \mathcal{O} \times(\tau, \infty), \\
			u=0 \text { on } \Gamma \times(\tau, \infty), \\
			u(x, \tau)=u_\tau(x), x \in \mathcal{O},
		\end{array}\right.
	\end{equation}
	where  $\nu>0$ is a  viscosity coefficient, $p$ represents the pressure term, $h\in L^2_{loc}(\mathbb{R};H)$, $\beta\in\mathbb{R}$.
	
	Now,  the operator $A: V \rightarrow V^{*}$ defined by $\langle A u, v\rangle=((u, v))$, and $D(A)=$ $\left(H^2(\mathcal{O})\right)^2 \cap V$. Then for any $u \in D(A), A u=-P
\triangle u$ is a Stokes operator, where $P$ is the orthoprojector from $\left(L^2(\mathcal{O})\right)^2$ onto $H$. We denote by $0<\lambda_1\leqslant\lambda_2\leqslant\cdots$ the
eigenvalues of $A$.
	Then the $2D$ stochastic Navier-Stokes equation can  be rewritten as:
	\begin{align}
		du=(\nu Au+B(u)+h(t))dt+\beta u_td\omega_{t},
	\end{align}
	where
	$$B:V\times V\rightarrow V^*, B(u,v)=-P[(u\cdot\nabla)v], B(u):= B(u,u).$$
	Finally, we have to required that  the $|h|^2$ is exponentially integrable(see Definition \ref{growth}) and  $\lim_{t\rightarrow -\infty}e^{-\eta t}\int_{-\infty}^{0}e^{\eta_1
r}|h(r+t)|^2dr=0$ for any $\eta>0$ and $\eta_1<\lambda_1(\frac{\nu}{2}-\epsilon),\epsilon<\frac{\nu}{2}$  to construct the bounded, tempered, random attracting set.  Hence it
generates a continuous cocycle $\Phi$, and  it has a unique $\mathcal{D}$-pullback random attractor $\mathcal{A}(\tau,\omega)$. we next illustrate the corresponding conditions for
this system and it can be found in \cite[Example 3.3]{MR2833734}.
	\begin{proof}
		$(A1^\prime)$ is obvious for $\nu Au+B(u)+h(t)$ since $A$ is linear and $B$ is bilinear. Let $t\in\mathbb{R},v\in V$, then $\eta=0$ and $\rho(v)=\|v\|^4_{L^4(\mathcal{O})}$
in $(A2^\prime)$.  $g(t)=0, f(t)=C|h(t)|^2,\alpha=2,c=\frac{\nu}{2}$ in $(A3^\prime)$. $\varpi=2$ in $(A4^\prime)$. In addition, $V\subseteq H$ is compact. For any $T>0$, then
Corollary \ref{well-posedness-o-1} and Theorem \ref{dependece-2} show that stochastic Naiver-Stokes system has a unique solution in $C([\tau,\tau+T],H)\cap L^2([\tau,\tau+T],V)$.
Therefore, Theorem \ref{rds} and Theorem \ref{thm:random attractor-n} tell us that the solution  of stochastic Navier-Stokes system can generate a continuous cocycle $\Phi$ and it
has a unique $\mathcal{D}$-pullback random attractor.
	\end{proof}	
	
\end{ex}
\appendix

\section{Autonomous  and Nonautonomous RDS}\label{app:rds}
In this subsection, we recall the theory of autonomous and nonautonomous random dynamical systems(RDS). For more details we refer to Refs. \cite{MR1374107,MR1427258,MR2927390}.
Let $(X,d)$ be a complete separable metric space and $(\Omega,\F,\P,\{\t_{t}\}_{t\in\R})$ be a metric dynamical system, i.e.\ $(t,\o)\mapsto\theta_{t}(\o)$
is ($\mcB(\R)\otimes\mcF,\mcF)$-measurable, $\theta_{0}=$ id, $\theta_{t+s}=\theta_{t}\circ\theta_{s}$ and $\theta_{t}$ is $\P$-preserving for all
$s,t\in\R$.
\begin{de}
	A family $\{D(\o)\}_{\o\in\Omega}(or~ \{D(\tau,\o)\}_{\tau\in\mathbb{R},\o\in\Omega})$  of subsets of $X$ is called  a random set if it is  $\o\to d(x,D(\o))(or~ \o\to
d(x,D(\tau,\o)))$ is measurable for each $x\in X(or~x\in X~and~\tau\in\mathbb{R})$.
\end{de}
\begin{de}\label{growth}
	A continuous function  $f : \R\to\R$ is said to be
	\begin{enumerate}
		\item tempered if $\lim\limits _{r\to-\infty}f(r)e^{\eta r}=0$ for all $\eta > 0$;
		\item exponentially integrable if $ |f| \in L_{loc}^{1}(\R;\R_{+})$ and $\int_{-\infty}^{t}|f(r)|e^{\eta r}dr < \infty$ for all $t\in\R$, $\eta > 0$.
	\end{enumerate}
\end{de}
\begin{de}
	A random set $D(\omega)(or~D(\tau,\omega))$ is tempered if there exists a  random variable $r(\omega)(or~r(\tau,\omega))\in\mathbb{R}^{+}$ such that
	$D(\omega)\in B_X(0,r(\omega))(or~D(\tau,\omega)\in B_X(0,r(\tau,\omega)))$ and $r(\theta_t\o)(or~ r(\tau,\theta_t\o))$ is tempered, where $B_X(0,r(\omega))$ is a ball in space
$X$ with 	center $0$ and radius $r(\omega)$.
\end{de}
In this paper, we denote the set  of random tempered sets by $\mathcal{D}$.
\subsection{Autonomous case}
\begin{de}
	Let $(\Omega,\mathcal{F},\mathbb{P},\{\theta_{t}\}_{t\in\mathbb{R}})$ be a metric dynamical system and $X$ be a separable Banach space. We call the mapping $\varphi$  continuous
random dynamical system on $X$ if $\varphi$ has the following properties:
	\begin{enumerate}
		\item $\varphi(0,\omega,\cdot)=Id_{X}$ for all $\omega\in\Omega$;
		\item $\varphi(t+\tau,\omega,x)=\varphi(t,\theta_{\tau}\omega,\varphi(\tau,\omega,x))$ for all $t,\tau\in \mathbb{R}^{+}$ and $x\in X$;
		\item $\varphi(t,\omega,\cdot): V\rightarrow V$  is continuous for all $t\in \mathbb{R}^{+}$ and $\omega\in\Omega$;
		\item the mapping $\varphi:\mathbb{R^{+}}\times\Omega\times V\rightarrow V$ is
$(\mathcal{B}(\mathbb{R}^{+})\otimes\mathcal{F}\otimes\mathcal{B}(X),\mathcal{B}(X))$-measurable.
	\end{enumerate}
\end{de}
\begin{de}
	A family  $B=(B(\omega))_{\omega\in\Omega}\subset X$  is called pullback absorbing set for $\mathcal{D}$ if
	$$\varphi(t,\theta_{-t}\omega,D(\theta_{-t}\omega))\subset B(\omega),\quad t\geq T(D,\omega)$$
	for any $D\in \mathcal{D}$ and $\omega\in\Omega$, where $T(D,\omega)$ is called absorption time.
\end{de}
For two subsets $A,B\subseteq X$ we introduce a Hausdorff semi-distance:
\[
dist(A,B):=\begin{cases}
	\sup\limits _{a\in A}\inf\limits _{b\in B}d(a,b), & \text{ if }A\ne\emptyset;\\
	\infty, & \text{ otherwise.}
\end{cases}
\]
\begin{de}
	Let $\mathcal{A}=\{A(\omega)\}$ be a random attractor for random dynamical system $\varphi$ if
	\begin{enumerate}
		\item $\mathcal{A}(\omega)$ is a compact set for all $\omega\in\Omega.$
		\item for any $\omega\in\Omega$ and $t\in \mathbb{R}^{+}$,
		$$\varphi(t,\omega,\mathcal{A}(\omega))=\mathcal{A}(\theta_{t}\omega).$$
		\item for any $D\in\hat{\mathcal{D}},\omega\in\Omega$,
		$$\lim_{\mathbb{T}^{+}\ni t\rightarrow \infty}dist(\varphi(t,\theta_{-t}\omega,D(\theta_{-t}\omega)),\mathcal{A}(\omega))\rightarrow 0.$$
	\end{enumerate}
\end{de}
For autonomous case, we have the following theorem to construct random attractors, more details can found in Ref. \cite{MR1427258}.
\begin{thm}\label{attrctor-a}
	Let  $\varphi$ is a random dynamical system and has a compact pullback absorbing set $B$ in $\mathcal{D}$. Then $\varphi$ has a unique random attractor
$\mathcal{A}=\{\mathcal{A}(\omega)\}$ for $\mathcal{D}$ as follows:
	$$\mathcal{A}(\omega)=\bigcap_{s\geq T(B,\omega)\in \mathbb{T}^{+}}\overline{\bigcup_{t\geq s}\varphi(t,\theta_{-t}\omega,B(\theta_{-t}\omega))},\quad \forall \omega\in\Omega.$$
\end{thm}
\subsection{Nonautonomous RDS}
\begin{de}
	We call the mapping $\Phi:  \mathbb R^+ \times \mathbb R \times \Omega \times X
	\rightarrow  X$  a continuous cocycle on $X$ over $\mathbb R$ and  $(\Omega,\mathcal{F},\mathbb P,(\theta_t)_{t\in\mathbb R})$  if  for  all
$\omega\in\Omega,s,t\in\mathbb{R}^{+}$, the following conditions hold
	\begin{itemize}
		\item [(i)]   $\Phi (\cdot, \tau, \cdot, \cdot): \mathbb{R}^+ \times \Omega \times X
		\to X$ is
		$(\mathcal B (\R^+)   \times \mathcal F \times \mathcal B(X), \
		\mathcal B(X))$-measurable;
		
		\item[(ii)]    $\Phi(0, \tau, \omega, \cdot) $ is the identity on $X$;
		
		\item[(iii)]    $\Phi(t+ s, \tau, \omega, \cdot) =
		\Phi(t,  \tau +s,  \theta_{s} \omega, \cdot)
		\circ \Phi(s, \tau, \omega, \cdot)$;
		
		\item[(iv)]    $\Phi(t, \tau, \omega,  \cdot): X \to  X$ is continuous.
	\end{itemize}
\end{de}
Let   $\mathcal D$   be
a family  of  nonempty subsets of $X$:
$$
\label{defcald}
{\mathcal D} = \{ D =\{ D(\tau, \omega ) \subseteq X: \
D(\tau, \omega ) \neq \emptyset,  \
\tau \in \mathbb R, \
\omega \in \Omega\} \}.
$$
We need give  a notion of  $\mathcal D$-complete  solutions  for $\Phi$.
\begin{de}
	\label{comporbit}
	Let $\mathcal D$ be a   family  of
	nonempty  subsets of $X$  as above.

	(i)  we call a  mapping $\psi: \R \times \R \times \Omega$
	$\to X$ a complete orbit (solution)  for  $\Phi$ if for each  $t \in \R^+ $,
	$s, \tau \in \R$ and $\omega \in \Omega$,  the following holds:
	$$
	\Phi (t,  \tau +s, \theta_{s} \omega,
	\psi (s, \tau, \omega) )
	= \psi (t +  s, \tau, \omega ).
	$$
	In addition,	if  there exists $D=\{D(\tau, \omega): \tau \in \R,
	\omega \in \Omega \}\in \mathcal D$ such that
	$\psi(t, \tau, \omega)$ belongs to
	$D ( \tau +t, \theta_{ t} \omega )$
	for every  $t \in \R$, $\tau \in \R$
	and $\omega \in \Omega$, then we call $\psi$ a
	$\mathcal D$-complete orbit (solution)  of $\Phi$.

	(ii)   A mapping $\xi:  \R   \times \Omega$
	$\to X$ is called a complete quasi-solution  of $\Phi$ if for every  $t \in \R^+ $,
	$\tau \in \R$ and $\omega \in \Omega$,  the following holds:
	$$
	\Phi (t,  \tau,  \omega,
	\xi (\tau, \omega) )
	= \xi (  \tau +t,  \theta_t \omega ).
	$$
	If, in  addition,    there exists $D=\{D(\tau, \omega): \tau \in \R,
	\omega \in \Omega \}\in \mathcal D$ such that
	$\xi (\tau, \omega)$ belongs to
	$D ( \tau,  \omega )$
	for all   $\tau \in \R$
	and $\omega \in \Omega$, then $\xi$ is called a
	$\mathcal D$-complete   quasi-solution  of $\Phi$.	
\end{de}
Consider the  metric dynamical systems $(\tilde\Omega_1,  \{\theta_{1,t}\}_{t \in \mathbb R})$
and
$(\tilde\Omega_2, \mathcal{F}_2, \mathbb P,  \{\theta_{2,t}\}_{t \in \mathbb R})$, which definition is similar to Remark \ref{att r}.
\begin{de}
	\label{asycomp}
	$\Phi$ is called $\mathcal{D}$-pullback asymptotically
	compact in $X$ if
	for all $\omega_1 \in \tilde\Omega_1$ and
	$\omega_2 \in \tilde\Omega_2$,    the sequence
	$$
	\label{asycomp1}
	\{\Phi(t_n, \theta_{1, -t_n} \omega_1, \theta_{2, -t_n} \omega_2,
	x_n)\}_{n=1}^\infty \mbox{  has a convergent  subsequence  in }   X
	$$
	whenever
	$t_n \to \infty$, and $ x_n\in   B(\theta_{1, -t_n}\omega_1,
	\theta_{2, -t_n} \omega_2 )$   with
	$\{B(\omega_1, \omega_2): \omega_1 \in \tilde\Omega_1, \ \omega_2 \in \tilde\Omega_2
	\}   \in \mathcal{D}$.
\end{de}
\begin{rem}\label{att r}
	In this paper, we require that $\tilde\Omega_1=\mathbb{R}, \theta_{1,t}s=s+t,s,t\in\mathbb{R}$,  $(\tilde\Omega_2, \mathcal{F}_2, \mathbb P,  \{\theta_{2,t}\}_{t \in \mathbb R})$
is $(\Omega,\mathcal{F},\mathbb{P},\{\theta_{t}\}_{t\in\mathbb{R}})$.
\end{rem}
\begin{rem}\label{asymptotically compact}
	In particular, for any $B\in\mathcal{D}$, if there exists a compact random absorbing set $K(\o_1,\o_2)$ for $B$, i.e.
	$$\lim_{t\rightarrow -\infty}\sup_{x\in B(\theta_{1,-t}\o_1,\theta_{2,-t}\o_2)}d(\Phi(t,\theta_{1,-t}\o_1,\theta_{2,-t}\o_2,x),K(\o_1,\o_2))=0$$
	holds for all $\o_1\in\tilde{\Omega_1}$ and $\o_2\in\tilde{\Omega_2}$.	Thus, for any $\o_1\in\tilde{\Omega_1}$ and $\o_2\in\tilde{\Omega_2}$, we have
	$$\lim_{n\rightarrow \infty}\sup_{x_n\in B(\theta_{1,-t_n}\o_1,\theta_{2,-t_n}\o_2)}d(\Phi(t_n,\theta_{1,-t_n}\o_1,\theta_{2,-t_n}\o_2,x_n),K(\o_1,\o_2))=0$$
	holds for any sequences $\{x_n\}$ and $\{t_n\}$ with $t_n\rightarrow \infty$ as $n\rightarrow \infty$. According to the compactness of $K$,  then the set
$\{\Phi(t_n,\theta_{1,-t_n}\o_1,\theta_{2,-t_n}\o_2,x_n)\}$ is asymptotically compact. It means that the cocycle  $\Phi$ must be asymptotically compact if the cocycle $\Phi$ the
precompact and has a bounded absorbing set.
\end{rem}
\begin{de}
	\label{defomlit}
	Let $B=\{B(\omega_1, \omega_2): \omega_1 \in \tilde\Omega_1, \ \omega_2  \in \tilde\Omega_2\}$
	be a family of nonempty subsets of $X$.
	For each $\omega_1 \in \tilde\Omega_1$ and
	$\omega_2 \in \tilde\Omega_2$,  let
	\begin{align*}\label{omegalimit}
		\Omega (B, \omega_1, \omega_2)
		= \bigcap_{\tau \ge 0}
		\  \overline{ \bigcup_{t\ge \tau} \Phi(t, \theta_{1,-t} \omega_1,
			\theta_{2, -t} \omega_2, B(\theta_{1,-t} \omega_1, \theta_{2,-t}\omega_2  ))}.
	\end{align*}
	Then
	the  family
	$\{\Omega (B, \omega_1, \omega_2): \omega_1 \in \tilde\Omega_1, \omega_2 \in \tilde\Omega_2 \}$
	is called the $\Omega$-limit set of $B$
	and is denoted by $\Omega(B)$.
\end{de}
\begin{thm}{\cite[Theorem 2.23]{MR2927390} }
	\label{att}	Let $\mathcal D$ a  neighborhood closed  collection of some
	families of   nonempty subsets of
	$X$   and $\Phi$  be a continuous   cocycle on $X$
	over $(\tilde\Omega_1,  \{\theta_{1,t}\}_{t \in \mathbb R})$
	and
	$(\tilde\Omega_2, \mathcal{F}_2, \mathbb P,  \{\theta_{2,t}\}_{t \in \mathbb R})$.
	Then
	$\Phi$ has a  $\mathcal D$-pullback
	attractor $\mathcal A$  in $\mathcal D$
	if and only if
	$\Phi$ is $\mathcal D$-pullback asymptotically
	compact in $X$ and $\Phi$ has a  closed
	measurable (w.r.t. the $P$-completion of $\mathcal{F}_2$)
	$\mathcal D$-pullback absorbing set
	$K$ in $\mathcal D$.
	The $\mathcal D$-pullback
	attractor $\mathcal A$   is unique   and is given  by,
	for each $\omega_1  \in \tilde\Omega_1$   and
	$\omega_2 \in \tilde\Omega_2$,
	\begin{equation}\label{attform1}
		\mathcal A (\omega_1, \omega_2)
		=\Omega(K, \omega_1, \omega_2)
		=\bigcup_{B \in \mathcal D} \Omega(B, \omega_1, \omega_2)
	\end{equation}
	\begin{equation}\nonumber
		=\{\psi(0, \omega_1, \omega_2): \psi \mbox{ is a  }  \mathcal D {\rm -}
		\mbox{complete orbit of } \Phi\} .
	\end{equation}
\end{thm}

\renewcommand{\theequation}{A\arabic{equation}}
\setcounter{equation}{0}
%
%

\end{document}